\documentclass[11pt]{amsart}

\usepackage{amssymb}
\usepackage{amsthm, bm}
\usepackage{amsmath}
\usepackage{graphicx}
\usepackage{comment}
\usepackage[usenames,dvipsnames]{xcolor}
\usepackage[margin = 1in]{geometry}
\usepackage{enumerate}
\usepackage[toc,page]{appendix}
\usepackage{lineno}
\usepackage{color}
\usepackage{hyperref}

\definecolor{mygreen}{rgb}{0.1,0.75,0.2}

 \newtheorem{thm}{Theorem}[section]
 
 \newtheorem{lem}[thm]{Lemma}
 \newtheorem{prop}[thm]{Proposition}

 \theoremstyle{definition}
 \newtheorem{defn}{Definition}

 \theoremstyle{remark}
 \newtheorem{rem}{Remark}

 \numberwithin{equation}{section}

\newcommand{\la}{\langle}
\newcommand{\ra}{\rangle}
\newcommand{\pt}{\partial}
\newcommand{\eps}{\varepsilon}

\newcommand{\sij}{{ij}}
\newcommand{\ud}{\,\mathrm{d}}
\newcommand{\el}{E_{\mathrm{els}}}
\newcommand{\elc}{\mathcal{C}_{\mathrm{els}}}
\newcommand{\hel}{\hat{E}_{\mathrm{els}}(\bm \varphi; \mathbf u)}
\newcommand{\helpsi}{\hat{E}_{\mathrm{els}}(\bm \psi; \mathbf u)}
\newcommand{\eg}{E_{\Gamma}}
\newcommand{\heg}{\hat{E}_{\Gamma}( \varphi_1; u_1)}
\newcommand{\hegu}{\hat{E}_{\Gamma}(u_1)}
\newcommand{\egg}{E_{\Gamma_e}}
\newcommand{\eggc}{\mathcal{C}_{\Gamma_e}}
\newcommand{\hegg}{\hat{E}_{\Gamma_e}(\varphi_1; u_1)}
\newcommand{\het}{\hat{E}_{\mathrm{total}}(\bm \varphi; \mathbf u)}
\newcommand{\hetpsi}{\hat{E}_{\mathrm{total}}( \bm \psi; \mathbf u)}
\newcommand{\epss}{\eps^*}
\newcommand{\epsv}{\eps_{\varphi}}
\newcommand{\sigs}{\sigma^*}
\newcommand{\sigvv}{\sigma_{\varphi}}
\newcommand{\fcc}{f\circ u_1^+}
\newcommand{\backg}{\backslash\Gamma}

\newcommand{\8}{\infty}

\newcommand{\F}{F}
\newcommand{\ff}{\dot{\Lambda}^{s+\frac12}}
\newcommand{\ffn}{\dot{\Lambda}^{s}}

\newcommand{\ptf}{(-\pt_{xx})^{\frac{1}{2}}}

\newcommand{\well}{W}

\begin{document}

\title{Mathematical validation of the Peierls--Nabarro model for edge dislocations}

\author{Yuan Gao}
\address{Department of Mathematics, Hong Kong University of Science and Technology, Clear Water Bay, Hong Kong \&
Department of Mathematics, Duke University,
  Durham NC 27708, USA}
\email{yg86@duke.edu}

\author{Jian-Guo Liu}
\address{Department of Mathematics and Department of Physics, Duke University,
  Durham NC 27708, USA}
\email{jliu@math.duke.edu}

\author{Tao Luo}
\address{Department of Mathematics, Purdue University, West Lafayette IN 47907, USA}
\email{luo196@purdue.edu}

\author{Yang Xiang}
\address{Department of Mathematics, Hong Kong University of Science and Technology, Clear Water Bay, Hong Kong}
\email{maxiang@ust.hk}

\date{\today}

\begin{abstract}
In this paper, we perform mathematical validation of the Peierls--Nabarro  (PN) models, which  are multiscale models of dislocations that incorporate the detailed dislocation core structure. We focus on the static and dynamic PN models of an edge dislocation. In a PN model, the total energy includes the elastic energy in the two half-space continua and a nonlinear potential energy across the slip  plane, which is always infinite.
We rigorously establish the relationship between the PN model in the full space and the reduced problem on the slip plane in terms of both governing equations and energy variations. 
The shear displacement jump is determined only by the reduced problem on the slip plane while the displacement fields in the two half spaces are determined by linear elasticity.
We establish the existence and sharp regularities of classical solutions in Hilbert space. 
 For both the reduced problem and the full PN model, we  prove that a static solution is a global minimizer in perturbed sense. We also show that there is a unique classical, global in time solution of the dynamic PN model.
\end{abstract}

\maketitle

\section{Introduction}
Materials defects such  as dislocations are important structures in materials science. Dislocations are line defects in crystalline materials and the major carriers of plastic deformation \cite{HirthLothe}.
Many plastic and mechanical behaviors of materials are associated with the energetic and dynamic properties of dislocations. Understandings of these properties also form a basis for the development of many novel materials with  robust performance.

 As a line defect, a dislocation has a small region (called the dislocation core region) of heavily distorted atomistic structures with shear displacement jump along a slip plane; as illustrated Fig. \ref{fig:pn}.  The dislocation core structures play essential roles in determining the energetic and dynamic properties of dislocations, such as the dislocation line {energies} and the critical  stresses for the motion of dislocations.  
 The classical dislocation  theory \cite{HirthLothe} regards the dislocation core as a singular point so that the solution can be solved explicitly  based on  the linear elasticity theory. Although the classical dislocation theory works well outside the dislocation core  regions, it gives nonphysical singularities within the dislocation cores.   
One way to precisely describe the dislocation core structure on the continuum level is  the Peierls--Nabarro (PN) model \cite{PN1,PN2,p4}, which is a multiscale continuum model that incorporates the atomistic effect by introducing a nonlinear potential  describing the atomistic interaction across the slip plane of the dislocation. 
\begin{figure}[htbp]
\centering
    \includegraphics[width=.6\linewidth]{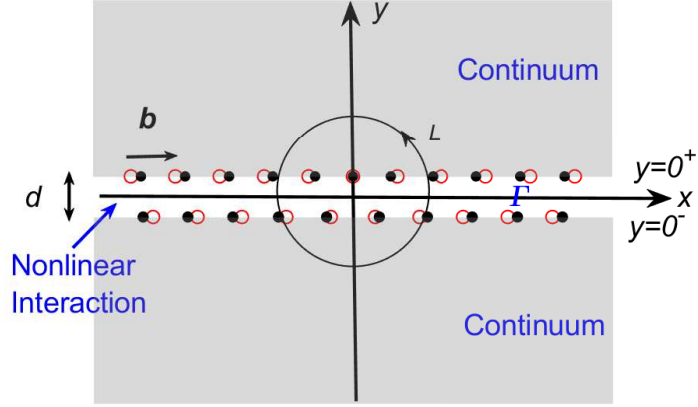}
    \caption{Schematic illustration of the PN model for an edge dislocation. The dislocation  locates along the $z$ axis with $+z$ direction, and its slip plane is the $y=0$ plane.  $\mathbf b$ is the Burgers vector and $d$ is the interplanar distance in the direction normal to the slip plane. The black dots and red circles show the locations of atoms of the two atomic planes $y=0^+$ and $y=0^-$ in the lattice with the dislocation and in the reference states before elastic deformation, respectively, based on a simple cubic lattice. The Burgers vector enclosed by a loop $L$ enclosing the dislocation is $\mathbf b_L=\oint_L \ud \mathbf u$. }
    \label{fig:pn}
\end{figure}

This paper focuses on the PN model for a straight edge dislocation \cite{HirthLothe} described below.
Assume that the dislocation  is located along the $z$ axis with $+z$ direction, and the slip plane of the dislocation
 is $\Gamma:=\{(x,y); y=0\}$.  Since the elastic field is uniform along the dislocation ($z$ direction), this problem is reduced to a two-dimensional problem in the $xy$ plane.
 In the PN model, the whole space is divided by the dislocation slip plane $\Gamma$ into two elastic continua $y>0$ and $y<0$ in which linear elasticity theory holds, and the two continua are connected by a nonlinear atomistic potential force across the slip plane $\Gamma$, see Fig.~\ref{fig:pn}.
The  displacement field $\mathbf u(x,y):=(u_1(x,y), u_2(x,y))$ has a shear displacement  jump across the slip plane $\Gamma$, i.e., $u_1$ is discontinuous across $\Gamma$.  

Dislocations are characterized by their Burgers vectors $\mathbf b$, which measure 
the direction and total magnitude of the shear displacement jump. The Burgers vector is defined as $\mathbf b=\oint_L \ud \mathbf u$, where $\mathbf u$ is the displacement vector and $L$ is any loop that encloses the dislocation line with counterclockwise orientation; see Fig.~\ref{fig:pn}. For the edge dislocation in Fig. \ref{fig:pn}, since the displacement $\mathbf u$ is differentiable in the half planes $y>0$ and $y<0$, the amplitude of $\mathbf b_L=(b_L,0)$ is $ b_L=\int_{\alpha}^\beta \left(-u_1'(x,0^+)+u_1'(x,0^-) \right)\ud x$, where $\alpha$ and $\beta$ are the intersection points of the loop $L$ with the $x$-axis. When the loop $L$ becomes infinitely large, $ b_L \to[-u_1'(+\8,0^+)+u_1'(+\8,0^-)]-[-u_1'(-\8,0^+)+u_1'(-\8,0^-)]=:b$.

Denote $\mathbf u^+$, $\mathbf u^-$ as the  {displacement fields} for the two half-spaces $\{(x,y); y>0\}$ and $\{(x,y); y<0\}$ respectively.
We impose  the following symmetric assumption
\begin{equation}\label{symsym}
u_1^+(x, 0^+)=-u_1^-(x, 0^-),\quad
u_2^+(x, 0^+)=u_2^-(x, 0^-).
\end{equation}
 and then the  far field boundary conditions at $y=0$ become
 \begin{equation}\label{BC}
 u_1^+(-\infty, 0^+)=\frac{b}{4}, \quad u_1^+(+\infty, 0^+)=-\frac{b}{4}.
 \end{equation}
For this edge dislocation, as  illustrated in Fig.~\ref{fig:pn} (based on simple cubic lattice for the locations of atoms near the slip plane $y=0$),  the reference states of the elastic deformation in the two half-space continua are different since there is an extra upper half plane of atoms located at $x=0$ in the upper space $y>0$. 
The shear displacement jump across the slip plane, or the disregistry, is
\begin{flalign}\label{disregistry}
\phi(x):=u_1^+(x, 0^+)-u_1^-(x, 0^-)+\frac{b}{2},
\end{flalign}
with the property
\begin{flalign}
\phi(-\infty)=b, \ \ \phi(+\infty)=0.
\end{flalign}
This means that away from the dislocation, we still have the perfect crystal lattice.
Note that the term $b/2$ in Eq.~\eqref{disregistry} is to account for the disregistry (relative shift) between the reference states in the upper and lower half spaces  in the direction of the Burgers vector.

 In the classical dislocation model \cite{Volterra}, the density of the magnitude of Burgers vector $\rho(x)=-\phi'(x)=b \delta(x)$, where $\delta(x)$ is the Dirac delta function, leads to singular displacement, strain and stress fields. Whereas in the PN model, the density of Burgers vector $\rho(x)=-\phi'(x)$ is a smoothed profile due to the incorporation of the nonlinear atomistic interaction across the slip plane. 
More precisely, the displacement fields are determined by minimizing the total energy $E(\mathbf u)$ including the elastic energy 
\begin{equation}\label{Eels}
E_{\mathrm{els}}(\mathbf u):=\frac12\int_{\mathbb{R}^2\backg} \sigma:\eps \ud x \ud y
\end{equation}
 in the two half spaces separated by the slip plane and nonlinear  misfit energy across the slip plane due to nonlinear atomistic interactions
\begin{equation}\label{Emisn}
 {E_\mathrm{mis}}(\mathbf u):=\int_{\Gamma} \gamma(\phi) \ud x.
\end{equation}
The misfit energy density $\gamma$ depends on the disregistry $\phi$ across the slip plane \eqref{disregistry} and is called the $\gamma$-surface~\cite{p4}. Using the boundary symmetry conditions in Eq.~\eqref{symsym}, we write the $\gamma$-surface as a function of $u_1^+$
 $$\gamma(\phi) =\gamma (u_1^+-u_1^-+b/2)  =\gamma (2u_1^++b/2) =:\well(u_1^+)$$
  for convenience of notation in the analysis. In a general one-dimensional model, $\gamma(\phi)$ is a bounded multi-well potential with period $b$ (period $b/2$ for $\well(v)$), and any minimum of it describes the perfect lattice.

The most important feature of the minimizing problem for the PN model above is that the shear displacement jump $u_1^+(x, 0^+)-u_1^-(x, 0^-)$  across the slip plane can be determined by a reduced one-dimensional model, i.e. a fractional Laplacian equation with a nonlinear potential force
\begin{equation}\label{Gamma_eq}
 -\frac{2G}{(1-\nu)\pi}{\mathrm{P.V.}} \int_{-\infty}^{+\infty} \frac{\pt_x u_1^+(s)}{x-s} \ud s=\well'(u_1^+), \quad x\in \mathbb{R},
\end{equation}
with boundary condition \eqref{BC},  where  $G$ is the shear modulus and $\nu$ is the Possion ratio.

 As a solvable example, the nonlinear potential takes the form of sinusoidal function \cite{PN1,PN2}, which phenomenologically reflects the lattice periodicity \cite{Frenkel},
\begin{equation}\label{spec-w}
\well(u_1)=\frac{Gb^2}{4\pi^2d}(1+\cos \frac{4\pi u_1}{b}),
\end{equation} 
 where $d$ is a constant indicating the interplanar distance in the direction normal to the slip plane; see Fig \ref{fig:pn}.
 A nontrivial solution is
 $u_1^+(x)=-\frac{b}{2\pi}\tan^{-1}\frac{x}{\zeta}$, where $\zeta=\frac{d}{2(1-\nu)}$ and $2\zeta$ is the core width of the dislocation, with the far field decay rate 
\begin{equation}\label{decay-r}
 u_1^+(x)\pm\frac{b}{4}\sim \frac{b\zeta}{2\pi x}\quad  \text{ as }x\to \pm\8.
\end{equation} 
Then by solving the linear elastic equation in the two half spaces, we obtain the special solution to the full system \cite{HirthLothe,xiang_cicp}
   \begin{equation}\label{vec-s}
   \begin{array}{l}
   u_1(x,y)=\frac{b}{2\pi}\left[-\tan^{-1}\frac{x}{y\pm\zeta}+\frac{xy}{2(1-\nu)(x^2+(y\pm\zeta)^2)}\right],\\
   u_2(x,y)=-\frac{b}{2\pi}\left[\frac{1-2\nu}{4(1-\nu)}\log(x^2+(y\pm\zeta)^2)+
   \frac{x^2-y^2+\zeta^2}{4(1-\nu)(x^2+(y\pm\zeta)^2)}\right],
   \end{array}
   \end{equation}
We call this solution the elastic extension of $u_1(x)$; see Theorem \ref{def_els}.

 Instead of the elastic extension, a  scalar model using harmonic extension  to obtain scale  solution in the two half spaces
 $$\tilde{u}(x,y)=-\frac{b}{2\pi}\tan^{-1}\frac{x}{y\pm\zeta}$$ 
 plays important role in studying dislocations.
 For the mathematical analysis for the static solution to the reduced PN model \eqref{Gamma_eq} and the  scalar model are well studied in \cite{Caff, XC2005, XC2015, Pala2013}.  In \cite{XC2005}, for a general misfit potential $\gamma$ with $C^{2,\alpha}$ regularity, \textsc{Cabr\'e and Sol\`a-Morales}  (i) established the existence  (unique up to translation) of  monotonic solutions with $C^{2, \alpha}$ regularity; (ii) recovered the sharp decay rate \eqref{decay-r} for the bistable profile; (iii) proved the bistable profile is a global minimizer relative to perturbations in $[-\frac{b}{4}, \frac{b}{4}]$ for the total energy $E(\tilde{u})$ for the scalar model. 
 In \cite{Pala2013}, \textsc{Dipierro, Palatucci and Valdinoci}  directly worked on the nonlocal equation \eqref{Gamma_eq} and improved  the global minimizer  result (iii) by removing the above $[-\frac{b}{4}, \frac{b}{4}]$-restriction on perturbations.
  Similar results for  the existence, regularities, and uniqueness of nonlocal equation with general fractional Laplacian $(-\Delta )^{\frac{s}{2}}$  for exponent $s\in(0,1)$    are obtained by \textsc{Cabr\'e and Sire}~\cite{XC2015}.

For the dynamic PN model,  viscosity solutions of the analogy scalar model which is
a heat equation with a  dynamic boundary condition are studied by \textsc{Fino, Ibrahim and Monneau} \cite{Fino}. The authors established existence and uniqueness of the viscosity solution to the scaler model using comparison principle for second order equations and the harmonic extension, which works only for scalar solutions.

With the same assumption on the above general misfit potential $\gamma$ with $C^{2,\alpha}$ regularity, we summarize the main results in this paper as follows.
\begin{enumerate}[(i)]
\item For the reduced nonlocal equation \eqref{Gamma_eq}, we obtain a sharp regularity result $u_1\notin \dot{H}^{\frac12}(\mathbb{R})$, $u_1\in \dot{H}^{s}(\mathbb{R})$ for any $s>\frac12$ (see Proposition \ref{regu}).
\item We extend $u_1$ to the two half spaces as $\mathbf u$ using elastic extension (see Theorem \ref{def_els}) and obtain the corresponding sharp regularity $\mathbf u\notin \dot{\Lambda}^{1}_\Gamma(\mathbb{R}^2)$, $\mathbf u \in \ff_\Gamma(\mathbb{R}^2)$ for any $s>\frac12$ (See \eqref{space} for homogeneous Hilbert space $\ffn(\mathbb{R}^2)$). 
\item
Combining above two facts (i) and (ii), we obtain the well-possedness of the static dislocation model in the full space (see Theorem \ref{thm_steady}).
\item We establish the energy connections between the reduce model on $\Gamma$ and the full system in $\mathbb{R}^2$ in the perturbed sense (see Theorem \ref{energy-re}), and then use it to prove the static solution (unique upto translations) to the full system is the global minimizer of the  total energy in the perturbed sense (see Theorem \ref{thm_mini}).
\item For the dynamic PN model in the full system, we obtain the global classical solution under quasi-static assumptions in the two half spaces.
\end{enumerate}

To the best of our knowledge, almost no study in the literature that explores the true three-dimensional vector field solutions to the static and dynamic PN models.  In fact, the vector-field displacement is essential to determine  long-range elastic interactions associated with dislocations and dislocation core structures. In contrast to the harmonic extension, we do not have maximal principle for the elastic extension. Indeed, the displacement fields $\mathbf u$ in \eqref{vec-s} has a $\ln r$ growth rate at far field, which is same as that of the two-dimensional stream function in  fluids or the two-dimensional electrostatic potential.

This paper only focus on the analysis for a  single edge dislocation model. There are vast literature in mathematical and physics studying dislocations and related problems and we only list a few here.
For example, 
some different physical models have been generalized and applied to calculate dislocation line energy, critical stress for the motion of dislocations, energy of grain boundaries which  consist of arrays of dislocations, and structure and morphology of bilayer materials with dislocations, e.g. \cite{p4, Kaxiras, p5, p6, p7, xiang_cicp, p38, p40, p8, p9, p42}. Convergence from atomistic model to the PN model with the $\gamma$-surface in bilayer materials has been proved \cite{Luo2018}.
There are also some results for other dislocation dynamics models, e.g., \cite{disdy1} proved    short time existence of a level set dislocation dynamics model \cite{Xiang2003}, and
convergence from PN models to larger scale models for a dislocation particle system, slow motion and other properties were analyzed \cite{Leoni1,Leoni2,DFV,Fino,Mon1,Monneau1,Leoni3,DPV,PV0,PV1,Fonseca}.  Some other techniques used for nonlocal equations rising from epitaxial surfaces were presented in \cite{Lu1, Lu2}.

The remaining sections of this paper is organized as follows. In Sec. \ref{sec2}, we first derive the reduced system and prove its sharp regularities; see Sec. \ref{cal_sigma12} and Sec. \ref{u1_bc} separately. Then we  establish the connection between the reduced system and the full system by the elastic extension in Theorem \ref{def_els} and use it to obtain the well-posedness of the full system in Theorem \ref{thm_steady}. In Sec. \ref{sec2.2}, we first obtain the connections between the reduced energy and the total energy in Theorem \ref{energy-re}. Then we prove the static solution (unique upto translations) to the full system is the global minimizer in perturbed sense; see Theorem \ref{thm_mini}. Sec. \ref{sec3} is devoted to prove existence of the global classical solution to the dynamic PN model.

\section{Well-posedness for static PN model}\label{sec2}
We investigate  solutions to the static PN model by first deriving the Euler--Lagrange equation which corresponds to  critical points of the total energy of the PN model. To understand connections between  solutions to the full space and solutions to the reduced system on slip plane, we establish the elastic extension analogue to the harmonic extension for the scalar model. More precisely, we will obtain the classical solution to the reduced system with sharp regularities and the  classical solution to the full system with corresponding regularities.

For simplification of notations, we will use $u_1^\pm(x)=u_1^\pm(x,0^\pm)$, $u_2^\pm(x)=u_2^\pm(x,0^\pm)$, $\sigma^\pm(x)=\sigma^\pm(x,0^\pm)$, etc.

\subsection{Total energy and Euler--Lagrange equation}
In the PN model, the two half spaces separated by the slip plane of the dislocation are assumed to be linear elastic continua, and the two half spaces are connected by a nonlinear potential
energy across the slip plane that incorporates  atomistic interactions, see Fig.~\ref{fig:pn}.
 The total energy  is
\begin{equation}\label{E2.2}
E(\mathbf u):= E_\mathrm{els}(\mathbf u)+E_\mathrm{mis}(\mathbf u).
\end{equation}
Here $\mathbf u$ is the displacement vector. As described in the introduction, in this PN model for the edge dislocation along the $z$ axis, the crystal structure is uniform in the $z$ direction; as a result, the problem becomes a problem in the $xy$ plane and the displacement in $z$ direction with $u_3=0$. The energy $E(\mathbf u)$ is the energy per unit length along the dislocation, and the displacement vector can be written as $\mathbf u=(u_1,u_2)$.

The first term in the total energy in Eq.~\eqref{E2.2} is the elastic energy in the two half spaces defined in \eqref{Eels}.
Here $\varepsilon$ is the strain tensor:
\begin{equation}\label{strain}
 \eps_\sij=\frac{1}{2}(\pt_j u_i+\pt_i u_j),
  \end{equation}
 for $i,j=1,2, 3$,  (where $\partial_1=\partial_x:=\frac{\partial}{\partial x}$, $\partial_2=\partial_y:=\frac{\partial}{\partial y}$, and $\partial_3=\partial_z:=\frac{\partial}{\partial z}$,) $\sigma$ is the stress tensor:
\begin{equation}\label{constitutive}
\sigma_\sij=2G \eps_\sij+\frac{2\nu G}{1-2\nu} \eps_{kk}\delta_{ij},
\end{equation}
for $i,j=1,2,3$ (in an isotropic medium),   $\delta_{ij}=1$ when $i=j$ and $0$ otherwise, and $\sigma : \varepsilon=\sigma_{ij}\varepsilon_{ij}$. We have used the Einstein summation convention that $\eps_{kk}=\sum_{k=1}^3 \eps_{kk}=\sum_{k=1}^2 \eps_{kk}$ and $\sigma_{ij}\varepsilon_{ij}=\sum_{i,j=1}^3 \sigma_{ij}\varepsilon_{ij}= \sum_{i,j=1}^2 \sigma_{ij}\varepsilon_{ij}$.

The second term in the total energy in Eq.~\eqref{E2.2} is the misfit energy across the slip plane due to nonlinear atomistic interactions defined in \eqref{Emisn}, i.e.
\begin{equation}\label{Emis}
 {E_\mathrm{mis}}(\mathbf u):=\int_{\Gamma} \gamma(\phi) \ud x =\int_{\Gamma} \well(u_1^+) \ud x.
 \end{equation}
 For the analysis of the PN model for an edge dislocation in this paper, we assume that the nonlinear potential $W$ satisfies for some $\alpha\in(0,1)$
\begin{equation}\label{potential}
\begin{aligned}
 &\well\in C^{2, \alpha}(\mathbb{R}), \vspace{1ex} \\ &W(v)>W\left(-\frac{b}{4}\right)=W\left(\frac{b}{4}\right), \ {\rm for}\ v \in \left(-\frac{b}{4}, \frac{b}{4} \right), \vspace{1ex} \\
 &  W''\left(\pm\frac{b}{4}\right)>0;
 \end{aligned}
 \end{equation}
see \eqref{spec-w} for example.
\begin{rem}
We remark that if we assume further $W$ is an even function, then the solution $u_1^+$ to \eqref{Gamma_eq} will be a unique odd function with the center $u_1^+(0,0^+)=0$, which is the case that an extra upper half plane of atoms locates at $x=0$ as illustrated in Fig. \ref{fig:pn}. Without this additional assumption, the solution $u_1$ and the corresponding $\mathbf u$ are unique upto translations; see  Theorem \ref{thm_steady}.
\end{rem}

The equilibrium structure of the edge dislocation is obtained by minimizing the total energy in Eq.~\eqref{E2.2} subject to the boundary condition at the slip plane given in Eq.~\eqref{BC}.
However, it is known that for a straight dislocation, the strain $\varepsilon$ and the stress $\sigma$ decay with rate $1/r$ at far field where $r$ is the distance to the dislocation, thus the elastic energy $E_\mathrm{els}$ is infinity \cite{HirthLothe}; see Remark \ref{rem-c} below. To be precise, we define the perturbed elastic energy of $\mathbf u$ with respect to any perturbation fields $ \bm\varphi  \in C^\infty(\mathbb{R}^2\backslash \Gamma; \mathbb{R}^2)$ and $ \bm\varphi $ has compact support in some $B(R)$ as
\begin{equation}\label{per1tm}
\begin{aligned}
\hel:=& \int_{\mathbb{R}^2\backslash \Gamma} \frac{1}{2} (\eps_u + \eps_\varphi): (\sigma_u + \sigvv)  -\frac{1}{2} \varepsilon:\sigma~\ud x\\
    =& \int_{\mathbb{R}^2\backslash \Gamma} \frac{1}{2}  [(\epsv)_{ij}(\sigvv)_{ij}+(\epsv)_{ij}(\sigs)_{ij}+ (\epss)_{ij}(\sigvv)_{ij} ]   ~\ud x\\
    =& \el(\bm\varphi ) + \elc(\mathbf u, \bm\varphi )
\end{aligned}
\end{equation}
where the cross term
\begin{equation}\label{crossels}
  \elc(\mathbf u, \bm\varphi ):= \int_{\mathbb{R}^2\backslash \Gamma} \frac{1}{2}  (\epsv:\sigma_u+ \eps_u:\sigvv )   ~\ud x= \int_{\mathbb{R}^2\backslash \Gamma} \frac{1}{2}  [(\epsv)_{ij}(\sigma_u)_{ij}+ (\eps_u)_{ij}(\sigvv)_{ij} ]   ~\ud x,
\end{equation}  
   where $\eps_u, \sigma_u$ and $\epsv, \sigvv$ are the stain and stress tensors corresponding to $\mathbf u$ and $\bm\varphi $ respectively. 
Then the perturbed total energy is defined as
\begin{equation}\label{perEtotal}
\het:= \hel+\int_\Gamma \well(u_1+\varphi_1)-\well(u_1) \ud x.
\end{equation}
\begin{defn}\label{minimizer}
We call a function $\mathbf u$ a global minimizer of total energy $E$ if it satisfies 
\begin{equation}
\het\geq 0
\end{equation}
for any perturbation $ \bm\varphi  \in C^\infty(\mathbb{R}^2\backslash \Gamma; \mathbb{R}^2)$ supported in some $B(R)$ satisfying
\begin{equation}\label{bcphi}
\varphi_1^+(x, 0^+)=-\varphi_1^-(x, 0^-), \, \varphi_2^+(x, 0^+)=\varphi_2^-(x, 0^-).
\end{equation} 
\end{defn}

 We have the following lemma for the Euler--Lagrange equation with respect to the total energy  $E(\mathbf u)$.

\begin{lem}\label{Lem2.2}
Assume that
 $\mathbf u\in C^2(\mathbb{R}^2\backslash \Gamma) $ satisfying boundary conditions \eqref{symsym} and \eqref{BC}  is a  minimizer of the total energy $E$ in  the sense of Definition \ref{minimizer}. Then $\mathbf u$ satisfies the  Euler--Lagrange equation
\begin{equation}\label{maineq}
\begin{aligned}
&\Delta \mathbf u+\frac{1}{1-2\nu }\nabla( \nabla\cdot \mathbf u)=0 \quad {\rm in} \ \mathbb{R}^2\backslash\Gamma, \vspace{1ex}\\
&\sigma_{12}^+ +\sigma_{12}^-=\well'(u_1^+) \quad {\rm on} \ \Gamma,\\
&\sigma_{22}^+ =\sigma_{22}^- \quad {\rm on}\ \Gamma.
\end{aligned}
\end{equation}
\end{lem}

\begin{proof}
From the Definition \ref{minimizer} of minimizer, we calculate the variation of energy in terms of a perturbation with compact support in an arbitrary ball $B(R)$.
For any $ \mathbf v\in C^\infty(B(R)\backslash \Gamma)$ such that $ \mathbf v$ has compact support in $B(R)$ and satisfies \eqref{bcphi}, we consider the perturbation $\delta \mathbf v$ where $\delta$ is a small real number. We denote  $\varepsilon:=\varepsilon(\mathbf u)$, $\sigma:=\sigma(\mathbf u)$ and $\varepsilon_1:=\varepsilon(\mathbf v)$, $\sigma_1:=\sigma(\mathbf v)$.
Then we have that
\begin{equation}\label{tem2.8}
\begin{aligned}
~&\lim_{\delta\to 0}\frac{1}{\delta} (E(\mathbf u+\delta \mathbf v)-E(\mathbf u))\\
=& \int_{B(R)\backslash \Gamma}\frac{1}{2}(\sigma_1:\varepsilon+ \sigma:\varepsilon_1)\ud x \ud y  +\int_{[-R, R]} \well'(u_1^+)v_1^+ \ud x\\
=&\int_{B(R)\backslash \Gamma}\sigma:\varepsilon_1\ud x \ud y  +\int_{[-R, R]} \well'(u_1^+)v_1^+ \ud x\\
=&\int_{B(R)\backslash \Gamma}\sigma:\nabla\mathbf v\ud x \ud y  +\int_{[-R, R]}\well'(u_1^+)v_1^+ \ud x\\
=&-\int_{B(R)\backslash \Gamma}\partial_j\sigma_{ij} v_i\ud x \ud y  +\int_{{[-R, R]}\cap \{y=0^+\}}\sigma_{ij}^+ n_j^+ v_i^+ \ud x \\
&\quad + \int_{{[-R, R]}\cap\{y=0^-\}}\sigma_{ij}^- n_j^- v_i^- \ud x
 +\int_{[-R, R]} \well'(u_1^+)v_1^+ \ud x\geq 0\\
\end{aligned}
\end{equation}
where we used the property that $\sigma$  and $\nabla \cdot \sigma$ are locally integrable in $\{y>0\}\cup\{y<0\}$ when carrying out the integration by parts, and
 the outer normal vector of the boundary $\Gamma$ is
 $\mathbf n^+$ (resp. the $\mathbf n^-$) for the upper  {half-plane} (resp. lower half-plane). Similarly, taking perturbation as $-\mathbf v$, we have
 \begin{equation}
\begin{aligned}
~&\lim_{\delta\to 0}\frac{1}{\delta} (E(\mathbf u-\delta \mathbf v)-E(\mathbf u))\\
=&\int_{B(R)\backslash \Gamma}\partial_j\sigma_{ij} v_i\ud x \ud y  -\int_{{[-R, R]}\cap \{y=0^+\}}\sigma_{ij}^+ n_j^+ v_i^+ \ud x \\
&\quad - \int_{{[-R, R]}\cap\{y=0^-\}}\sigma_{ij}^- n_j^- v_i^- \ud x
 -\int_{[-R, R]} \well'(u_1^+)v_1^+ \ud x\geq 0.
\end{aligned}
\end{equation}
 Hence
 \begin{align*}
& -\int_{B(R)\backslash \Gamma}\partial_j\sigma_{ij} v_i\ud x \ud y  +\int_{{[-R, R]}\cap \{y=0^+\}}\sigma_{ij}^+ n_j^+ v_i^+ \ud x \\
&\quad + \int_{{[-R, R]}\cap\{y=0^-\}}\sigma_{ij}^- n_j^- v_i^- \ud x
 +\int_{[-R, R]} \well'(u_1^+)v_1^+ \ud x= 0
 \end{align*}
  Noticing that $\mathbf n^+=(0,-1)$ and $\mathbf n^-=(0,1)$, we have
 \begin{equation}
 \begin{aligned}
&\int_{\{y=0^+\}}\sigma_{ij}^+ n_j^+ v_i^+ \ud x + \int_{\{y=0^-\}}\sigma_{ij}^- n_j^- v_i^- \ud x \\
=& \int_{\{y=0^+\}}-\sigma_{22}^+  v_2^+ \ud x+ \int_{\{y=0^-\}}\sigma_{22}^-  v_2^- \ud x+ \int_{\{y=0^+\}}-\sigma_{12}^+  v_1^+ \ud x+ \int_{\{y=0^-\}}\sigma_{12}^-  v_1^- \ud x.
 \end{aligned}
 \end{equation}
Since  $v_1^+(x)=-v_1^-(x)$ and $v_2^+(x)=v_2^-(x)$. Hence due to the arbitrariness of $R$, we conclude that the minimizer  $\mathbf u$ must  satisfy
\begin{equation}\label{tem2.10}
\begin{aligned}
&\int_{\Gamma}\left[\sigma_{12}^+ + \sigma_{12}^- -\well'(u_1^+)\right] v_1^+ \ud x=0,\\
&\int_\Gamma \left(\sigma_{22}^+-\sigma_{22}^-\right) v_2^+ \ud x =0,\\
&\int_{\mathbb{R}^2\backslash \Gamma} (\nabla\cdot \sigma) \cdot\mathbf v~ \ud x\ud y =0
\end{aligned}
\end{equation}
for any $\mathbf v\in  C^\infty(B(R)\backslash\Gamma)$ and $ \mathbf v$ has compact support in $B(R)$,
 which leads to the Euler--Lagrange equation \eqref{maineq}. Here we have written the equation $\nabla\cdot \sigma=0$ in $\mathbb{R}^2 \backslash \Gamma$ as the first equation of \eqref{maineq} in terms of the displacement $\mathbf u$, using the constitutive relation in \eqref{constitutive} and the definition of the strain tensor in  \eqref{strain}.
\end{proof}



\subsubsection{Working Space}
To better understand the sharp working space for the PN dislocation model, let us first see an example for classical nonlinear potential below. 

\begin{rem}\label{rem-c}
Recall the special solution $u_1^+(x)=-\frac{b}{2\pi}\tan^{-1}\frac{x}{\zeta}$ for the reduced model \eqref{Gamma_eq} when the nonlinear potential is \eqref{spec-w}.  Using this solution of the reduced problem on $\Gamma$, the solution of the full PN model, i.e., the Euler--Lagrange equation \eqref{maineq} with the boundary conditions \eqref{symsym} and \eqref{BC}, is shown in \eqref{vec-s}. The stress tensor is then
  \begin{equation}\label{sp-sigma}
    \sigma=\frac{Gb}{2\pi(1-\nu)}\left(
    \begin{array}{ccc}
   -\frac{3y\pm2\zeta}{x^2+(y\pm\zeta)^2}+\frac{2y(y\pm\zeta)^2}{[x^2+(y\pm\zeta)^2]^2} &
   \frac{x}{x^2+(y\pm\zeta)^2}-\frac{2xy(y\pm\zeta)}{[x^2+(y\pm\zeta)^2]^2}&0 \\
    \frac{x}{x^2+(y\pm\zeta)^2}-\frac{2xy(y\pm\zeta)}{[x^2+(y\pm\zeta)^2]^2} &
    -\frac{y}{x^2+(y\pm\zeta)^2}+\frac{2x^2y}{[x^2+(y\pm\zeta)^2]^2}&0 \\
  0&0&-\frac{2\nu(y\pm\zeta)}{x^2+(y\pm\zeta)^2}
      \end{array}
      \right),
  \end{equation}
  where  $+\zeta$ applies for $y>0$ while $-\zeta$ applies for $y<0$. Note that in this case, the disregistry across $\Gamma$ defined in \eqref{disregistry} is $\phi(x)=2u_1^+(x)+\frac{b}{2}=-\frac{b}{\pi}\tan^{-1}\frac{x}{\zeta}+\frac{b}{2}$ and the density of the Burgers vector is $\rho(x)=-\phi'(x)=\frac{b}{\pi}\frac{\zeta}{x^2+\zeta^2}$.
\end{rem}

From this example, 
 the Fourier transform of $u_1^+(x)=-\frac{b}{2\pi}\tan^{-1}\frac{x}{\zeta}$ in tempered distributional sense is $-\frac{ib}{2|\xi|}e^{-|\zeta \xi|}$. Thus we can show
 \begin{equation}\label{spec_s}
 \|u_1^+\|^2_{\dot{H}^{s}(\mathbb{R})}= \frac{b^2\Gamma(2s-1)}{4\pi(2\zeta)^{2s-1}} \,\text{ for } s>\frac12;\quad  \|u_1^+\|_{\dot{H}^{\frac12}(\mathbb{R})}=+\8,
 \end{equation}
 where $\Gamma(2s-1)$ is the Gamma function. So we want to study $u_1^+(x,0^+)\in H^{\frac12+\eps}$ for $\eps>0$.
 
For real number $s>0$ and intege $m\geq 0$, define the homogeneous Sobolev space
\begin{equation}
\dot{H}^{s,m}(\mathbb{R}^2\backslash\Gamma):= \{ u;\,\, (-\pt_{xx})^{\frac s2}\pt_y^m  u \in L^2(\mathbb{R}^2\backslash\Gamma)\}
\end{equation} 
with standard semi-norm $\|\cdot\|_{\dot{H}^{s,m}}(\mathbb{R}^2\backslash\Gamma).$
  Therefore it is natural to define spaces for $s\geq 1$
 \begin{equation}\label{space}
 \ffn_\Gamma(\mathbb{R}^2):=\{\mathbf u\in \dot{H}^{s-m,m} (\mathbb{R}^2\backslash\Gamma), \,0\leq m\leq [s],\, u_1^+(x, 0^+)=-u_1^-(x, 0^-), \, u_2^+(x, 0^+)=u_2^-(x, 0^-)\},
 \end{equation}
 where $[s]$ represents the integer part of $s.$ 
 Define the semi-norm for $\mathbf u\in\ffn_\Gamma(\mathbb{R}^2)$ as
 \begin{equation}
 \|\mathbf u\|^2_{\ffn_\Gamma(\mathbb{R}^2)}:= \sum_{m=0}^{[s]} \|\mathbf u\|_{\dot{H}^{s,m}(\mathbb{R}^2\backslash\Gamma)}^2.
 \end{equation}

 It is easy to check the example above belongs $\ff_\Gamma(\mathbb{R}^2)$ for $s>\frac12$ but $\sigma\sim \frac{1}{r}$ at far field implies $ \mathbf u\notin \dot{\Lambda}_\Gamma^1(\mathbb{R}^2)$. Due to the elastic continua is divided into two half spaces,  taking $m$ as an integer is to avoid technique complication for fractional derivatives in $y$ direction.  In this paper, we will see the working space for PN dislocation model is $\ff_\Gamma(\mathbb{R}^2)$  for  real number $s>\frac12$.
 
 To ensure we can take trace for any function $\mathbf u\in \ffn_\Gamma(\mathbb{R}^2)$, let us first prove the trace theorem for $\ffn_\Gamma(\mathbb{R}^2)$. The inverse trace theorem is proved in  Theorem \ref{def_els} by establishing the elastic extension.
 \begin{lem}[Trace Theorem]
 Given $\mathbf u\in \ffn_\Gamma(\mathbb{R}^2) $ for any $s\geq 1$, then the trace of $\mathbf u$, $u_i^\pm|_\Gamma\in \dot{H}^{s-\frac12}(\mathbb{R})$, $i=1,2$ and we have the estimate
 \begin{equation}\label{trace}
 \|u_i^\pm|_\Gamma\|_{\dot{H}^{s-\frac12}(\mathbb{R})} \leq \|\mathbf u\|_{\ffn_\Gamma(\mathbb{R}^2)}, \quad i=1,2.
 \end{equation}
 \end{lem}
\begin{proof}
Let $s\geq 1$ and denote $\hat{u}_1^+(\xi, y), \hat{u}_2^+(\xi, y)$ as the Fourier transform for  $u_1^+(x,y)$ and $u_2^+(x,y)$ with respect to $x$ by regarding them as tempered distributions. First, for the upper half plane and  any function $\mathbf u\in \ffn_\Gamma(\mathbb{R}^2) $ such that $\mathbf u$ vanishes as $y\to +\8$, we have
\begin{equation}
|\xi|^{2s-1} |\hat{u}_1^+(\xi, 0^+)|^2 = - 2|\xi|^{2s-1} \int_0^{+\8} \pt_y \hat{u}_1^+(\xi,y) \hat{u}_1^+(\xi, y) \ud y.
\end{equation}
Then by H\"older's inequality and Parserval's identity,
\begin{equation}
\|u_1^+\|^2_{\dot{H}^{s-\frac12}(\mathbb{R})} \leq 2\|(-\pt_{xx})^{\frac{s-1}{2}} \pt_y u_1^+\|  \|(-\pt_{xx})^\frac{s}{2}  u_1^+\|\leq \|\mathbf u\|^2_{\ffn_\Gamma(\mathbb{R}^2)}.
\end{equation}
This estimate holds also for $u_2^+$ and the lower half plane.
Thus by dense argument, we concludes \eqref{trace}.
\end{proof}

\subsection{Dirichlet to Neumann map}\label{cal_sigma12}
In this section, we study a representation in the sense that for given $u_1^\pm$ on $\Gamma$, we can uniquely determine the traction $(\sigma_{12}^\pm, \sigma_{22}^\pm)$ on $\Gamma$  using the  elasticity system in $\mathbb{R}^2\backslash \Gamma$. This is the Dirichlet to Neumann map for the linear elasticity system. As a consequence of  the Dirichlet to Neumann map we reduce the Euler--Lagrange equation \eqref{maineq} in $\mathbb{R}^2$ to a problem on $\Gamma$ (to be discussed in the next subsection). The following lemma gives the Dirichlet to Neumann map.
Note that P.V. denotes the Cauchy principal value of the integral.

\begin{lem}\label{lem2.2}
Assume that $\mathbf u\in \ff_\Gamma (\mathbb{R}^2)$ for some $s\geq \frac12$ satisfies the Euler--Lagrange equation \eqref{maineq}.
We have the following conclusions.
\begin{enumerate}[(i)]
\item (Fourier representation)
The solution $\mathbf u(x,y)$ in $\mathbb{R}^2$ can be represented entirely by $u_1^\pm(x,0^\pm)$ on $\Gamma$ as follows.
\begin{equation}\label{solu1}
\hat{u}_1^\pm(\xi,y)=\hat{u}_1^\pm(\xi,0^\pm)\left( 1-\frac{|\xi|}{2-2\nu}|y| \right)e^{-|\xi y|}, \,
\end{equation}
\begin{equation}\label{solu4}
\hat{u}_2^\pm(\xi,y)=-\frac{\hat{u}_1^+(\xi,0^+)}{2-2\nu} \left( (1-2\nu)\frac{i\xi}{|\xi|}+ i \xi |y| \right) e^{-|\xi y|}, 
\end{equation}
\item (Dirichlet to Neumann map) If $ u_1^+|_\Gamma\in \dot{H}^1(\mathbb{R})$ then
 $\sigma_{12}^\pm$ and $\sigma_{22}^\pm$ on $\Gamma$ are in $L^2(\mathbb{R})$ and can be expressed by
\begin{flalign}\label{sigmato}
&\sigma_{12}^+(x)=\sigma_{12}^-(x)=-\frac{G}{(1-\nu)\pi} \mathrm{P.V.} \int_{-\infty}^{+\infty} \frac{{u_1^+}'(s)}{x-s}\ud s,\\
&\sigma_{22}^+(x)=\sigma_{22}^-(x)=0. \label{sigmato1}
\end{flalign}
\item  If $ u_1^+|_\Gamma\in \dot{H}^1(\mathbb{R})$, then $\mathbf{u}$ also satisfies the elastic equation in whole space in the distributional sense, i.e. 
\begin{equation}
\nabla \cdot \sigma = \mathbf{0}, \quad \text{ in }\mathcal{D}'(\mathbb{R}^2).
\end{equation}
\end{enumerate}
\end{lem}

\begin{proof}
Step 1. We solve the elasticity problem, i.e., the first equation in \eqref{maineq},  by using the Fourier transform with respect to $x$.  Note that $u_1(x,y)$ is not in $L^2(\mathbb{R})$ for a fixed $y$ due to its asymptotic behavior in \eqref{BC}. Therefore, we take the Fourier transform for  $u_1(x,y)$ and $u_2(x,y)$ with respect to $x$ by regarding them as tempered distributions.  For notation simplicity, denote the Fourier transforms as $\hat{u}_1(\xi,y)$ and $\hat{u}_2(\xi,y)$.

Taking the Fourier transform with respect to $x$ in the first equation in \eqref{maineq}, we have
\begin{equation}\label{fequ}
\begin{aligned}
(1-2\nu)\partial_{yy}\hat{u}_1-(2-2\nu)\xi^2\hat{u}_1+i\xi \partial_y\hat{u}_2=0, \\
(2-2\nu)\partial_{yy}\hat{u}_2-(1-2\nu)\xi^2\hat{u}_2+i\xi \partial_y\hat{u}_1=0,
\end{aligned}
\end{equation}
in the tempered distributional sense.
Eliminating $\hat{u}_2$, we obtain an ODE for $\hat{u}_1$
\begin{equation}\label{ode1}
\partial_y^4\hat{u}_1-2 \xi^2 \partial_y^2 \hat{u}_1+\xi^4 \hat{u}_1=0.
\end{equation}
The eigenvalues are determined by the characteristic equation
$k^4-2\xi^2k^2+\xi^4=0$,
which has two double roots $k_1=k_2=\xi, \, k_3=k_4=-\xi$.

We first consider the lower plane $y<0$.
Since $\mathbf u\in \ff_\Gamma (\mathbb{R}^2)$,  the negative roots are not acceptable in this case, and the general solution of \eqref{ode1} is given by
\begin{equation}\label{gen1}
\hat{u}_1^-=(A^-+B^-|\xi|y)e^{|\xi|y}, \ y<0,
\end{equation}
where constants $A^-, B^-$ may depend on $\xi$ and will be determined later.
Similar analysis gives general solutions
\begin{equation}
\hat{u}_2^-=\frac{|\xi|}{i \xi}(C^-+D^-|\xi|y)e^{|\xi|y}, \ y<0,
\end{equation}
and in the upper plane $y>0$,
\begin{flalign}
\hat{u}_1^+=(A^+-B^+|\xi|y)e^{-|\xi|y}, \ y>0, \label{gen3}\\
\hat{u}_2^+=\frac{|\xi|}{i\xi}(C^+-D^+|\xi|y)e^{-|\xi|y}, \ y>0, \label{gen4}
\end{flalign}
where constants $C^-, D^-, A^+, B^+, C^+, D^+$ may depend on $\xi$ and will be determined later.

Step 2. Now we express those constants in terms of $A^+$ using Euler--Lagrange equation \eqref{maineq} and boundary symmetry \eqref{symsym}.
First by induction,  we have the following  identities
\begin{equation}\label{diff-id}
\begin{aligned}
\pt_y ^m(e^{-|\xi|y})& = (-|\xi|)^m e^{-|\xi|y},\\
\pt_y^m (-|\xi|ye^{-|\xi|y} )&= (-|\xi|)^m (m-|\xi|y)e^{-|\xi|y},\\
 \pt_y ^m(e^{|\xi|y}) &= |\xi|^m e^{|\xi|y},\\
\pt_y^m (|\xi|ye^{|\xi|y} )&= |\xi|^m (m+|\xi|y)e^{|\xi|y},
\end{aligned}
\end{equation}
for any $m\in \mathbb{N}^+.$
Then plugging the general solutions of $\hat{u}_1$ and $\hat{u}_2$ in \eqref{gen1}--\eqref{gen4} into \eqref{fequ}, we obtain the relations
\begin{flalign}
D^+=-B^+, \quad D^+ = \frac{1}{4\nu-3}(A^++ C^+) \label{cons1.1}\\
D^-=B^-, \quad D^- = \frac{1}{4\nu-3}(C^--A^-). \label{cons1.2}
\end{flalign}

Second, from $u_1(x,0^+)=-u_1(x,0^-)$ and $u_2(x,0^+)=u_2(x,0^-)$ in the boundary condition in \eqref{symsym} we have
\begin{equation}\label{Apm}
A^+=-A^-, \quad C^+=C^-,
\end{equation}
respectively.
Combining \eqref{Apm} with \eqref{cons1.1} and \eqref{cons1.2}, we have
\begin{equation}\label{BDpm}
B^+ =-B^-,\quad D^+ =D^-.
\end{equation}

Third, from the second boundary condition in \eqref{maineq},  {i.e.,} $\sigma_{22}^+=\sigma_{22}^-$ on $\Gamma$, and using \eqref{Apm} and \eqref{BDpm}, we have
$$
2(C^++ D^+)+ \frac{2\nu}{1-\nu}A^+=0.
$$
Using this equation and \eqref{cons1.2}, we obtain
\begin{equation}\label{temp2}
C^+=C^-=\frac{1-2\nu}{2-2\nu}A^+.
\end{equation}
Thus, all the constants in the general solutions of $u_1$ and $u_2$ in \eqref{gen1}--\eqref{gen4} can be determined by the constant $A^+$ by \eqref{cons1.1}--\eqref{temp2} as follows.
\begin{eqnarray}
B^+ =-B^-= \frac{1}{2-2\nu}A^+\\
D^+= D^-= -\frac{1}{2-2\nu}A^+.
\end{eqnarray}
Therefore we can further express the solutions as
\begin{equation}
\hat{u}_1=-A^+\left( 1+\frac{|\xi|}{2-2\nu}y \right)e^{|\xi|y}, \ y<0,
\end{equation}
\begin{equation}
\hat{u}_2=-\frac{A^+}{2-2\nu} \left( (1-2\nu)\frac{i\xi}{|\xi|}- i \xi y \right) e^{|\xi|y}, \ y<0,
\end{equation}
\begin{flalign}
\hat{u}_1=A^+\left( 1-\frac{|\xi|}{2-2\nu}y \right) e^{-|\xi|y}, \ y>0, \label{solu3}\\
\hat{u}_2=-\frac{A^+}{2-2\nu} \left( (1-2\nu)\frac{i\xi}{|\xi|}+ i \xi y \right) e^{-|\xi|y}, \ y>0.
\end{flalign}
Since we also have $\hat{u}_1^+(\xi,0)=A^+(\xi)$ by \eqref{gen3}, the conclusion (ii) follows.

Step 3. Using these obtained results, we can calculate that on $\Gamma$,
\begin{flalign}
&\hat{\sigma}_{12}^+=\hat{\sigma}_{12}^-=G(\pt_y \hat{u}_1^++ i\xi \hat{u}_2^+)=-\frac{G}{1-\nu}|\xi|A^+,\label{fsigmato}\\
&\hat{\sigma}_{22}^+=\hat{\sigma}_{22}^-=0.\label{fsigmato1}
\end{flalign}
Equation \eqref{sigmato1} follows directly from \eqref{fsigmato1}.
If further $ u_1^+|_\Gamma\in \dot{H}^1(\mathbb{R})$, using the definition of the Hilbert transform $H(f)(x)=\frac{1}{\pi}{\rm P.V.}\int^{+\infty}_{-\infty}\frac{f(s)}{x-s} \ud s$
and its Fourier transform $\widehat{H(f)}=-i{\rm sgn}(\xi)\hat{f}$, we obtain \eqref{sigmato} from \eqref{fsigmato}. This proves part (i).

Step 4. 
Given any test function $\bm\varphi \in C_c^\infty(\mathbb{R}^2)$, if $ u_1^+|_\Gamma\in \dot{H}^1(\mathbb{R})$, we calculate $\nabla\cdot \sigma$ in the weak sense. 
    \begin{align*}
      &\int_{\mathbb{R}^2} (\nabla \cdot \sigma) \cdot  \bm\varphi  \ud x\ud y =  \int_{\mathbb{R}^2} - \sigma :\nabla \bm\varphi  \ud x\ud y = \int_{\mathbb{R}^2\backslash\Gamma} - \sigma :\nabla \bm\varphi  \ud x\ud y\\
      =&\int_{\{y>0\}\cup\{y<0\}}\partial_j\sigma_{ij} \varphi_i\ud x \ud y  -\int_{\{y=0^+\}}\sigma_{ij}^+ n_j^+ \varphi_i \ud x - \int_{\{y=0^-\}}\sigma_{ij}^- n_j^- \varphi_i \ud x\\
      =&\int_{\{y>0\}\cup\{y<0\}}\partial_j\sigma_{ij} \varphi_i\ud x \ud y +\int_\Gamma (\sigma_{22}^+ - \sigma_{22}^- ) \varphi_2 +(\sigma_{12}^+ -\sigma_{12}^- ) \varphi_1 \ud x,
    \end{align*}
    where { we use the  {symmetry} property of $\bm\varphi $.} Since we have
    $\nabla\cdot \sigma =0$ in $\mathbb{R}^2\backslash\Gamma$, $\sigma_{22}^+|_{\Gamma} = \sigma_{22}^-|_{\Gamma}=0$ and $\sigma_{12}^+|_\Gamma=\sigma_{12}^-|_\Gamma$,  we obtain
    {
    \begin{equation}
     \int_{\mathbb{R}^2} (\nabla \cdot \sigma) \cdot \bm\varphi  \ud x\ud y =  -\int_\Gamma(\sigma_{12}^+ -\sigma_{12}^- ) \varphi_1 \ud x=0,
    \end{equation}
    which implies
    $$\nabla\cdot \sigma = \mathbf 0, \quad \text{ in }\mathcal{D}'(\mathbb{R}^2).
    $$
    This property explains  {that at} the equilibrium state the force acting on the elastic materials is zero everywhere. To determine the displacement field in the whole space, the staring point is free system without external force. Therefore the elastic equation $\nabla \cdot \sigma =0$ holds for the whole space in distribution sense. All the deformation comes from internal defect, which, in our case, is the present of single straight dislocation line defect. Hence the full system can be regarded as a linear elastic system for the upper and the lower plane connected by shear displacement jump on the interface, i.e. the slip plane $\Gamma$.    }
\end{proof}
The  lemma above  allows us to reduce the full system to the slip plane $\Gamma$, called the reduced system (see next subsection), by establishing the Dirichlet to Neumann map.


\subsection{Reduced problem on $\Gamma$ and its solvability}\label{u1_bc}
From Lemma \ref{lem2.2} part (i), we know  that the solution of the Euler--Lagrange equation \eqref{maineq} is entirely determined by the displacement $u_1^+(x)=u_1(x,0^+)$ on $\Gamma$. From Lemma \ref{lem2.2} part (ii),  $u_1^+$ on $\Gamma$ can be determined by the second equation in the Euler--Lagrange equation \eqref{maineq}. In this sense, the equation of $u_1^+$ on $\Gamma$ is called the reduced problem on $\Gamma$ and will be discussed in this subsection. How to determine the solution of the Euler--Lagrange equation \eqref{maineq} in $\mathbb{R}^2$ from the solution of the reduced problem will be discussed in the next subsection.


In fact, using Lemma \ref{lem2.2} part (ii) and the second equation of the Euler--Lagrange equation \eqref{maineq}, we know that the displacement $u_1$ on $\Gamma$, $u_1^+(x)=u_1^+(x, 0^+)$, is a solution of the
 nonlocal equation \eqref{Gamma_eq} on $\Gamma$; i.e.
\begin{equation}
 -\frac{2G}{(1-\nu)\pi}{\mathrm{P.V.}} \int_{-\infty}^{+\infty} \frac{\pt_x u_1^+(s)}{x-s} \ud s=\well'(u_1^+), \quad x\in \mathbb{R},
\end{equation}
with the  boundary condition in \eqref{BC}, i.e.
\begin{equation}\label{Gamma_bc}
\lim_{x\to-\infty}u_1^+(x)=\frac{b}{4}, \quad \lim_{x\to+\infty}u_1^+(x)=-\frac{b}{4},
\end{equation}
where $W$ is the nonlinear potential satisfying \eqref{potential}. This is the reduced problem on $\Gamma$.

The nonlocal term on the right-hand side of \eqref{Gamma_eq} is the Hilbert transform with a constant coefficient $-2G/(1-\nu)$, which can also be written in terms of the fractional Laplacian operator:
\begin{equation}
 H(v')(x)=\frac{1}{\pi}{\mathrm{P.V.}} \int_{-\infty}^{+\infty} \frac{v'(s)}{x-s} \ud s= (-\pt_{xx})^{\frac{1}{2}}v(x).
 \end{equation}
Recall that the fractional Laplacian operator
$  (-\pt_{xx})^{s}v(x):= C_s  {\mathrm{P.V.}} \int_{\mathbb{R}} \frac{v(x)-v(y)}{|x-y|^{1+2s}} \ud y$,
where $C_s$ is a normalizing constant to guarantee the symbol of the resulting operator is $|\xi|^{2s}$. 

We summarize the above results into the following proposition.
\begin{prop}[Reduced PN model]
 Assume that $\mathbf u \in \ff_\Gamma(\mathbb{R}^2)$ for some $s>\frac12$ is a solution of the  Euler--Lagrange
equation \eqref{maineq} with the boundary condition  \eqref{BC}. Then the displacement $u_1$ on $\Gamma$, $u_1^+(x)=u_1^+(x, 0^+)$, is a solution of the
 nonlocal equation   \eqref{Gamma_eq}  with boundary conditions \eqref{Gamma_bc} at $x=\pm\infty$.
\end{prop}

The existence result of  equation \eqref{Gamma_eq} subject to  far field boundary conditions \eqref{Gamma_bc} has been given by Theorem 2.4 in \cite{XC2015} (see also Theorem 1.2 in \cite{XC2005}), after rescaling of \eqref{Gamma_eq} into the form $2(-\pt_{xx})^{\frac{1}{2}}u_1^+=f(u_1^+)$ on $\Gamma$.

\begin{prop}[Solvability of Reduced model]\label{thm2.3}
Consider  the
 nonlocal equation   \eqref{Gamma_eq}  with boundary conditions \eqref{Gamma_bc}.
\begin{enumerate}[(i)]
\item
(Theorem 2.4 in \cite{XC2015}) There exists a bounded solution $u_1^+(x)$ (unique up to translations) 
such that $\pt_x u_1^+(x)<0$ in $\mathbb{R}$.
\item (Theorem 1.6 in \cite{XC2005}) The solutions satisfy the asymptotic behavior $|\mp\frac{b}{4}-u_1^+(x)|\sim \frac{1}{|x|}$ as $x\to \pm\8.$
\end{enumerate}
\end{prop}

Next we prove a sharp elliptic regularity result for $u_1^+(x)$ in the Sobolev space.
\begin{prop}\label{regu}
The solution $u_1^+(x)$ to nonlocal equation \eqref{Gamma_eq} with boundary condition \eqref{Gamma_bc} satisfies
 $u_1^+ \in \dot{H}^s(\mathbb{R})$ for any  $s>\frac12$.
\end{prop}
\begin{proof}
Step 1. We prove $u_1^+\in \dot{H}^1(\mathbb{R})$. From Proposition \ref{thm2.3} (ii) and Taylor expansion of $f$ at $u_1^+(\pm\8)=\mp \frac{b}{4}$,  $f(u_1^+(x))=W'(u_1^+(x))\sim \frac{1}{x}$ as $x\to \pm \8$. Therefore $f\circ u_1^+\in L^2(\mathbb{R})$. From  \eqref{Gamma_eq} and Parserval's identity we obtain
\begin{equation}
\||\xi|\hat{u}_1\|_{L^2(\mathbb{R})}= c_s\|\widehat{f\circ u_1^+}\|_{L^2(\mathbb{R})}\leq C,
\end{equation}
where $c_s$ is a rescaling constant and we concludes $u_1^+\in \dot{H}^1(\mathbb{R})$.

Step 2. We prove $u_1^+\in \dot{H}^s(\mathbb{R})$ for any $\frac12<s<1$  using the property $\fcc\sim \frac{1}{x}\in L^p(\mathbb{R})$ for any $p>1$. 
From \eqref{Gamma_eq}, we have for $\frac12<s<1$
\begin{equation}
(-\pt_{xx})^{\frac{s}{2}}u_1^+ = (-\pt_{xx})^{\frac{s-1}{2}}(\fcc)=(-\pt_{xx})^{-\frac{1-s}{2}}(\fcc),
\end{equation}
 where $(-\pt_{xx})^{-\frac{1-s}{2}} (\fcc)$ can be represented by the Riesz potential $I_{1-s}g:=c\int_{\mathbb{R}} |x-y|^{s}g (y) \ud y$ when $g\in L^p(\mathbb{R})$ such that $\frac32=\frac{1}{p}+s.$ Precisely,  from the Hardy-Littlewood-Sobolev Theorem for fractional integration \cite[p119, Theorem 1]{stein}
 \begin{equation}
 \|u_1^+\|_{\dot{H}^s}^2 = \|(-\pt_{xx})^{\frac{s}{2}}u_1^+\|_{L^2}^2=\|I_{1-s} (\fcc)\|_{L^2(\mathbb{R})}^2\leq c\|\fcc\|_{L^p(\mathbb{R})}^2.
 \end{equation}
 This concludes $u_1^+\in \dot{H}^s(\mathbb{R})$ for any $\frac12<s<1$.

Step 3. We prove $u_1^+\in \dot{H}^s(\mathbb{R})$ for any $1<s\leq\frac{3}{2}$. First we notice for any $s>0$,
\begin{equation}
\|\fcc\|_{\dot{H}^s} ^2= \frac{C_s}{2}\int\int \frac{|f(u_1^+(x))-f(u_1^+(y))|^2}{|x-y|^{1+2s}} \ud x \ud y \leq (\max f') \|u_1^+\|_{\dot{H}^s}^2.
\end{equation}
Therefore from Step 2, $u_1^+\in \dot{H}^s(\mathbb{R})$ for $\frac{1}{2}<s\leq 1$ implies $\fcc \in \dot{H}^s(\mathbb{R})$ for $\frac{1}{2}<s\leq 1$.
Then by \eqref{Gamma_eq} and Parserval's identity, we have for any $1<s\leq \frac{3}{2}$
\begin{equation}
\|u_1^+\|_{\dot{H}^s}^2 \leq \int_{\mathbb{R}} \xi^{s-\frac12} \hat{u}_1^+ \xi^{s-\frac12} \widehat\fcc \ud \xi \leq \|u_1^+\|_{\dot{H}^{s-\frac12}} \|\fcc\|_{\dot{H}^{s-\frac12}}<C
\end{equation}
due to both $u, \fcc \in \dot{H}^s(\mathbb{R})$ for $\frac12<s\leq 1$.

Step 4. In summary, for $s\in(\frac12, \frac32]$, we have $u_1^+\in \dot{H}^s(\mathbb{R})$ from Steps 1--3. By induction, we only show how to improve $s\in(\frac12, \frac32]$ to $s\in(\frac32, \frac52].$
Since $u\in \dot{H}^s$ for $s\in(\frac12,\frac32]$ we know $\fcc\in \dot{H}^s$ for $s\in(\frac12,\frac32]$. Thus by Parserval's identity, we have for $s\in (\frac12,\frac32]$
\begin{equation}
\||\xi|^{1+s}\hat{u}_1^+\|_{L^2}= \||\xi|^s \fcc\|_{L^2}\leq C,
\end{equation}
which concludes $u_1^+\in \dot{H}^{s}$ for any $s\in(\frac32, \frac52].$
\end{proof}

\subsection{Elastic extension in $\mathbb{R}^2\backslash \Gamma$ and its property}

 Analogue to the harmonic extension, we introduce an elastic extension that extends the function on $\Gamma$ to the two half spaces based on the elastic system in $\mathbb{R}^2\backslash\Gamma$. This is summarized into the following lemma.
\begin{thm}\label{def_els}
  Assume that $ \varphi_\Gamma\in {  \dot{H}^s(\mathbb{R})}$ for some real number $s\geq \frac12$. There exists a unique solution $\mathbf u\in \ff_\Gamma(\mathbb{R}^2)$ to the following elasticity problem in $\mathbb{R}^2\backslash \Gamma$:
\begin{equation}\label{elstemp}
\left\{
    \begin{array}{ll}
     & \Delta \mathbf u+\frac{1}{1-2\nu }\nabla( \nabla\cdot \mathbf u)=0 \quad {\rm in}\ \mathbb{R}^2\backslash \Gamma, \vspace{1ex}\\
     & u_1^+(x,0^+) =\varphi_\Gamma (x) \quad {\rm on} \ \Gamma, \vspace{1ex} \\
     & \sigma_{22}^+(x)=\sigma_{22}^-(x) \quad  \ {\rm on} \ \Gamma.
    \end{array}
  \right.
\end{equation}
And the solution satisfies the stability and regularity estimate
\begin{equation}\label{regu-es}
\|\mathbf u\|_{\ff_\Gamma(\mathbb{R}^2)} \leq C \|u_1^+\|_{\dot{H}^{s}(\mathbb{R})}.
\end{equation}
We call  solution $\mathbf u$ the elastic extension of $\varphi_\Gamma$. 
\end{thm}
\begin{proof}
Step 1. 
 It can be seen from  Lemma~\ref{lem2.2} part (i) that the solution $\mathbf u$ of the elastic system in $\mathbb{R}^2\backslash \Gamma$ is given by  Fourier representation \eqref{solu1}--\eqref{solu4} with the symmetric relations in \eqref{Apm}  and \eqref{BDpm}.  It shows that the solution  $\mathbf u$ exists and is uniquely determined by $u_1^+(x,0^+)=\varphi_\Gamma(x)$. 

 Step 2.  Regularity of $\mathbf{u}$ in $\mathbb{R}^2\backslash\Gamma$.
 By the Fourier representation formula \eqref{solu1}-\eqref{solu4}, we can take any derivatives w.r.t $y$. Recall  identities \eqref{diff-id}.
For any $0\leq m\leq [s+\frac12]$, we have for $y>0$
\begin{equation}\label{es-f}
(i \xi)^{s+\frac12-m} \pt_y^m \hat{u}^+_1(\xi, y)= (i \xi)^{s+\frac12-m} (-|\xi|)^m \hat{u}_1^+(\xi,0^+)\left( 1-\frac{m-|\xi|y}{2-2\nu} \right) e^{-|\xi|y}.
\end{equation}

For  $y>0$, from \eqref{es-f}, we estimate
\begin{equation}
\int_{\mathbb{R}^2_+} |(-\pt_{xx})^{\frac{s+\frac12-m}{2}} \pt_y^m u_1^+|^2 \ud x \ud y\leq C\int_{\mathbb{R}^2_+} |\xi|^{2s+1}|\hat{u}_1^+(\xi, 0^+)|^2 (1+|\xi|^2y^2)e^{-2|\xi|y} \ud \xi \ud y.
\end{equation}
Notice the identity
$$-\Big[\left( \frac{3+2|\xi|^2 y^2+2|\xi|y}{4 |\xi|} \right) e^{-2 |\xi| y}\Big]'=(1+|\xi|^2 y^2 )e^{-2|\xi|y}.$$ 
We have
$$-\left( \frac{3+2|\xi|^2 y^2+2|\xi|y}{4 |\xi|} \right) e^{-2 |\xi| y}\Big|_0^\8\leq \frac{3}{4|\xi|}.$$
Thus we obtain the uniform bound
\begin{equation}\label{es-tm}
\int_{\mathbb{R}^2_+} |(-\pt_{xx})^{\frac{s+\frac12-m}{2}} \pt_y^m u_1^+|^2 \ud x \ud y\leq C\int_{\mathbb{R}} |\xi|^{2s}|\hat{u}_1^+(\xi, 0^+)|^2  \ud \xi= C \|u_1^+\|_{\dot{H}^{s}}^2.
\end{equation}
for $y>0$ and any $0\leq m\leq [s+\frac12]$. This estimate also holds for $u_2$ or $y<0$.  Therefore we obtain the stability and regularity estimate
\begin{equation}
\|\mathbf u\|_{\ff_\Gamma(\mathbb{R}^2)}^2 \leq C \|u_1^+\|^2_{\dot{H}^{s}(\mathbb{R})}.
\end{equation}
\end{proof}
\begin{rem}
The elastic extension established in Theorem \ref{def_els} shows that for any function  $u_1^+|_\Gamma\in\dot{H}^s(\mathbb{R})$ with $s\geq \frac12$, there exists $\mathbf u\in \ff_\Gamma(\mathbb{R}^2)$ such that $u_1^+|_\Gamma$ is the trace of the first component of $\mathbf u$. This is an inverse trace theorem.
\end{rem}


\subsection{Existence, uniqueness and regularity  for the full PN model}
In this section, we establish the existence and uniqueness of the solution to Euler--Lagrange equation \eqref{maineq}, which is referred to as the full PN model, subject to the boundary conditions \eqref{symsym} and \eqref{BC}. After the reduced model on $\Gamma$ is solved in the last subsection, the solution of the full model is determined by an elastic extension from the solution on $\Gamma$.

We first have the following  mirror symmetry property for the displacement $\mathbf u$ in the whole space due to mirror symmetry boundary conditions in \eqref{symsym}.

\begin{lem}\label{equiv}
 Let $\mathbf u\in \ff_\Gamma(\mathbb{R}^2)$ for some $s\geq \frac12$ be the solution to the elasticity system in $\mathbb{R}^2\backslash \Gamma$ (the first equation in \eqref{maineq}).
Then $\mathbf u$  satisfies the mirror symmetry condition in the whole space
\begin{equation}\label{bc02}
u_1^+(x, y)=-u_1^-(x,-y),\quad u_2^+(x,y)=u_2^-(x,-y), \quad x\in \mathbb{R}, y\in \mathbb{R}^+.
\end{equation}
\end{lem}
The proof of this lemma directly follows the expressions of $u_1$ and $u_2$ in  \eqref{gen1}--\eqref{gen4} and the relationship of the coefficients in \eqref{Apm}  and \eqref{BDpm}  in the proof of Lemma \ref{lem2.2}.


After establishing the connection between solutions to reduced model and the full model by the elastic extension. We state the existence and regularity theorem below.
\begin{thm}\label{thm_steady}
Assume that the nonlinear potential $W$ satisfies \eqref{potential}. We have the following conclusions for  solutions to the full PN model.
\begin{enumerate}[(i)]
\item   There exists a classical solution (unique up to translations)   $\mathbf u\in \ff_\Gamma(\mathbb{R}^2)$ for any $s>\frac12$  to problem \eqref{maineq} with boundary conditions \eqref{symsym} and \eqref{BC}. Moreover, the solution $\mathbf u$ satisfies the symmetry condition in \eqref{bc02}.
\item The displacement component $u_1$ of the solution $\mathbf u$ on $\Gamma$ $u_1^+(x)=u_1^+(x, 0^+)$ is a classical solution in $\dot{H}^s$, for any $s>\frac12$ of the
 nonlocal equation   \eqref{Gamma_eq}  with boundary conditions \eqref{Gamma_bc} at $x=\pm\infty$.
\item The solution $\mathbf u$ can be regarded as the elastic extension of $u_1^+(x)$ on $\Gamma$ (which is the solution of the reduced problem of \eqref{Gamma_eq} and \eqref{Gamma_bc}) defined in Theorem~\ref{def_els}.
\end{enumerate}
\end{thm}
\begin{proof}
 We first apply Proposition~\ref{thm2.3}  to obtain the existence of a  solution (unique up to translations) $u_1^+(x)$ to problem \eqref{Gamma_eq} with boundary condition \eqref{Gamma_bc},
such that  $\pt_x u_1^+(x)<0$ in $\mathbb{R}$. The regularity $ {u}_1^+\in \dot{H}^s(\mathbb{R})$ for any $s>\frac12$ comes from Proposition \ref{regu}.
 Thus we have (ii). The solution $\mathbf u$ of the full PN model is obtained by the
elastic extension of
  $u_1^+(x)$ on $\Gamma$, which is uniquely determined, following Theorem~\ref{def_els}. From \eqref{regu-es},the regularity of $\mathbf u\in \ff_\Gamma(\mathbb{R}^2)$ for any $s>\frac12$ is ensured by the regularity of $u_1^+(x)$. Thus we conclude (iii) and the existence and regularity in (i). The symmetric property of $\mathbf u$ in (i) is a conclusion of Lemma~\ref{equiv}.
\end{proof}

\section{Global minimizer of total energy for the full system}\label{sec2.2}
 The goals in this section are to connect the total energy in the two half spaces with the reduced energy on $\Gamma$, which are both  infinite for a single straight dislocation, and then to prove the static solution $\mathbf u$ obtained in the last section is a global minimizer of the total energy $E(\mathbf u)$ in the sense of Definition \ref{minimizer}. Besides,  the first component trace $u_1|_\Gamma$ of the global minimizer $\mathbf u$ of the total energy is also a global minimizer of the reduced energy $E_\Gamma(u_1)$ defined in \eqref{E_G} below, vice versa.
To ensure all the energy estimates in this section meaningful, the natural idea is to compare the difference between $E(\mathbf u)$ and $E(\mathbf u+ \bm\varphi )$   such that the total {Burgers} vector for the perturbed displacement fields $\bm\varphi $ is zero. 
We will show  the  precise relation between reduced energy $E_\Gamma$ on slip plane $\Gamma$ and the total energy $E$ in \eqref{E2.2}  in Theorem \ref{energy-re}. 
We will see in the next section the reduced system on $\Gamma$ has its own gradient flow structure with respect to $E_\Gamma$. From now on, with slight abuse of notation, we use $u_1=u_1|_\Gamma:= u_1^+(x,0^+)$ to denote the restriction of the first component of vector fields $\mathbf{u}$ on the slip plane $\Gamma$.

\subsection{Energy relations between the full system and the reduced system}\label{sec2.5}
From the Dirichlet to Neumann map established in Section \ref{cal_sigma12}, we will reduce rigorously the total energy of the full PN model to an energy on the slip plane $\Gamma$.
Indeed, we define the free energy $E_\Gamma$ for the reduced system on the slip plane $\Gamma$ as
\begin{equation}\label{E_G}
E_\Gamma(u_1):=\int_\Gamma |(-\pt_{xx})^{\frac{1}{4}}u_1|^2 \ud x +\int_\Gamma \well(u_1) \ud x,
\end{equation}
which is finite for $u_1\in H^{\frac12}$. However for the static solutions obtained in the last section, $u_1|_\Gamma\in \dot{H}^s$ with $s>\frac12$ and
$E_\Gamma$ is infinite ; see also example in \eqref{spec_s}.
 Hence the idea is to  state the connection for  $E(\mathbf u+ \bm\varphi )-{E(\mathbf u)},$ where $\mathbf u$ is the static solution obtained in Theorem \ref{thm_steady}. Similar to \eqref{perEtotal}, we define the perturbation elastic energy of $\mathbf u$ on $\Gamma$ with respect to the trace $\varphi_1|_\Gamma$ of the perturbation $\bm\varphi $ as
\begin{equation}
\begin{aligned}
\hegg:=&  \int_\Gamma |(-\pt_{xx})^{1/4} (u_1+ \varphi_1)|^2 -  |(-\pt_{xx})^{1/4} u_1|^2 ~ \ud x\\
=& \int_\Gamma |(-\pt_{xx})^{1/4} \varphi_1|^2  ~ \ud x+2\int_\Gamma \varphi_1(-\pt_{xx})^{1/2} u_1  ~\ud x
\end{aligned}
\end{equation}
and the perturbed total energy on $\Gamma$ as
\begin{equation}\label{perEgamma}
\heg:= \hegg+\int_\Gamma \well(u_1+\varphi_1)-\well(u_1) \ud x.
\end{equation}
It is easy to  see the perturbed energy above is well-defined for any perturbations $\bm\varphi \in H^1(\mathbb{R}^2\backg)$ with $\varphi_1|_\Gamma\in H^\frac12(\mathbb{R})$ satisfying \eqref{bcphi}.

 We first summarize the energy connections in two cases. The proof of this theorem will be left to the end of this section after establishing some lemmas. 
\begin{thm}\label{energy-re}
Given $u_1|_\Gamma\in \dot{H}^{s}(\mathbb{R})$ for some  $s\geq \frac12$ and its elastic extension $\mathbf u\in \ff_\Gamma(\mathbb{R}^2)$, we consider the reduced energy $E_\Gamma(u_1)$  and the total energy $E(\mathbf u)$.
\begin{enumerate}[(i)]
\item If $u_1|_\Gamma\in H^{\frac12}(\mathbb{R})$, then
\begin{equation}
E(\mathbf u)= E_\Gamma(u_1)<\8.
\end{equation}
\item If $u_1|_\Gamma\notin \dot{H}^{\frac12}(\mathbb{R})$, then
\begin{equation}
E(\mathbf u)=+\8; \quad E_\Gamma(u_1)= +\8,
\end{equation}
and
 the relation of energies is stated in perturbed sense, i.e. for any perturbation $\varphi_1|_\Gamma\in H^{\frac12}(\mathbb{R})$ with $\bm\varphi \in H^1(\mathbb{R}^2\backslash\Gamma)$ being its elastic extension,  we have
\begin{equation}
\heg= \het,
\end{equation}
where $ \het$ is defined in \eqref{perEtotal} and $\heg$ is defined in \eqref{perEgamma}.
\end{enumerate}
\end{thm} 


First, we point out this result is standard if $u_1|_\Gamma\in H^{\frac12}(\mathbb{R})$, which yields a finite elastic energy; see Lemma \ref{reducelem} below. However, for the trace $u_1|_\Gamma\in \dot{H}^{s}(\mathbb{R})$ with some $s>\frac12$, which yields an infinite energy, we  will handle it later in Lemma \ref{lem2.8}.

Define the elastic part of $E_\Gamma(u)$ as $E_{\Gamma_e}(u):=\int_\Gamma |(-\pt_{xx})^{\frac{1}{4}} u_1|^2 \ud x.$
The following lemma shows that we can reduce the elastic energy in the two half spaces to the nonlocal energy $E_{\Gamma_e}$ on surface $\Gamma$. 
\begin{lem}\label{reducelem}
  Assume $\mathbf u\in \dot{\Lambda}^1_\Gamma (\mathbb{R}^2) $ is an elastic extension of $u_1|_\Gamma\in H^\frac12(\mathbb{R})$, then we have
\begin{equation}\label{tm-reduce}
   {E_\mathrm{els}}(u)=E_{\Gamma_e}(u_1).
\end{equation}
\end{lem}
\begin{proof}
  By dense argument, we only prove for $u\in \dot{\Lambda}^1_\Gamma (\mathbb{R}^2)\cap C^2(\mathbb{R}^2)$. We know
  \begin{align*}
    0&= \int_{\{y>0\}\cup\{y<0\}} \mathbf u \cdot ( \nabla\cdot \sigma) \ud x\ud y\\
    &=-\int_{\{y>0\}\cup\{y<0\}}\nabla \mathbf u: \sigma \ud x \ud y  +\int_{\{y=0^+\}}\sigma_{ij}^+ n_j^+ u_i^+ \ud x + \int_{\{y=0^-\}}\sigma_{ij}^- n_j^- u_i^- \ud x\\
    &= -\int_{\{y>0\}\cup\{y<0\}} \varepsilon: \sigma \ud x \ud y + \int_{\{y=0\}}(-\sigma_{22}^+  +\sigma_{22}^- ) u_2^+ \ud x+ \int_{\{y=0\}}(-\sigma_{12}^+  -\sigma_{12}^-)  u_1^+ \ud x\\
    &= -\int_{\{y>0\}\cup\{y<0\}} \varepsilon: \sigma \ud x \ud y + \int_{\{y=0\}}\frac{2G}{(1-\nu) } |(-\pt_{xx})^{\frac{1}{4}}u_1^+ |^2 \ud x,
  \end{align*}
  where we used $\sigma_{12}^+  +\sigma_{12}^- = -\frac{2G}{(1-\nu) } (-\pt_{xx})^{\frac{1}{2}}u_1^+$ due to Lemma \ref{lem2.2}. Without loss of generality, we set the physical constant $\frac{G}{(1-\nu)}$ to be $1$, so we obtain
  \begin{equation}
    2E_\mathrm{els}=\int_{\{y>0\}\cup\{y<0\}} \varepsilon: \sigma \ud x \ud y=  2\int_{\{y=0\}} |(-\pt_{xx})^{\frac{1}{4}}u_1^+ |^2 \ud x.
  \end{equation}
\end{proof}

Next we extend the lemma above to $u_1|_\Gamma\notin \dot{H}^{\frac12}$ with its elastic extension vector fields  $\mathbf u\in \ff_\Gamma$ with $s>\frac12$, which is the case (ii) in Theorem \ref{energy-re}. Since $u_1|_\Gamma\notin \dot{H}^{\frac12}$ implies an infinite reduced energy, 
instead of proving \eqref{tm-reduce} directly,  we compare the difference between $E(\mathbf u)$ and $E(\mathbf u+ \bm\varphi )$   such that  the perturbed displacement fields $\bm\varphi $ possessing finite energy.

\begin{lem}\label{lem2.8}
  Let $\mathbf u\in \ff_\Gamma(\mathbb{R}^2)$ with trace $u_1|_\Gamma\in \dot{H}^s(\mathbb{R})$ for some $s>\frac12$ be the static solution obtained in Theorem \ref{thm_steady}. Let $\varphi_1|_\Gamma$ be any $H^{\frac12}(\mathbb{R})$ perturbation  and let $\bm\varphi \in H^1(\mathbb{R}^2\backslash\Gamma)$ be the elastic extension of $\varphi_1|_\Gamma$. Then we have
  \begin{equation}
    \hel = \hegg
  \end{equation}
\end{lem}
\begin{proof}
  Recall the definition of energy functional $\el$ and $\egg$ and  the  cross term defined in \eqref{crossels}
  \begin{equation*}
  \elc(\mathbf u, \bm\varphi )= \int_{\mathbb{R}^2\backslash \Gamma} \frac{1}{2}  (\epsv:\sigma_u+ \eps_u:\sigvv )   ~\ud x= \int_{\mathbb{R}^2\backslash \Gamma} \frac{1}{2}  [(\epsv)_{ij}(\sigma_u)_{ij}+ (\eps_u)_{ij}(\sigvv)_{ij} ]   ~\ud x,
\end{equation*}  
   where $\eps_u, \sigma_u$ and $\epsv, \sigvv$ are the stain and stress tensor corresponding to $\mathbf u$ and $\bm\varphi $ respectively. Then we have
\begin{equation}\label{tm-p1}
  \begin{aligned}
    \hel
    =& \int_{\mathbb{R}^2\backslash \Gamma} \frac{1}{2} (\eps_u + \eps_\varphi): (\sigma_u + \sigvv)  -\frac{1}{2} \eps_u:\sigma_u ~\ud x\\
    =& \int_{\mathbb{R}^2\backslash \Gamma} \frac{1}{2}  [(\epsv)_{ij}(\sigvv)_{ij}+(\epsv)_{ij}(\sigma_u)_{ij}+ (\eps_u)_{ij}(\sigvv)_{ij} ]   ~\ud x\\
    =& \el(\bm\varphi ) + \elc(\mathbf u, \bm\varphi ).
  \end{aligned}
\end{equation}  
 Similarly, define
 the cross term
 $$\eggc(u_1, \varphi_1):= 2\int_\Gamma \varphi_1(-\pt_{xx})^{1/2} u_1  ~\ud x. $$
  Then for the energy functional $\eg$, we have
  \begin{align}\label{eqvE}
    \hegg =& \int_\Gamma |(-\pt_{xx})^{1/4} (u_1+ \varphi_1)|^2 ~- ~ |(-\pt_{xx})^{1/4} u_1|^2 ~ \ud x\\
   =& \int_\Gamma |(-\pt_{xx})^{1/4} \varphi_1|^2  ~ \ud x+2\int_\Gamma \varphi_1(-\pt_{xx})^{1/2} u_1  ~\ud x\\
    =& \egg(\varphi_1) + \eggc(u_1, \varphi_1).\nonumber
  \end{align}
  By Lemma \ref{reducelem}, $\el(\bm\varphi )= \egg(\varphi_1)$ due to $\varphi_1|_\Gamma\in H^{\frac12}(\mathbb{R}),$ so it remains to deal with the cross terms.

 Next, we claim the following relation for the cross terms.
 \begin{equation}\label{cross}
   \elc(\mathbf u, \bm\varphi ) = \eggc(u_1, \varphi_1).
 \end{equation}
In fact, from the symmetry of constitutive relation $\sigma_{ij}=C_{ijkl}\eps_{kl}$, we know
\begin{align*}
(\epsv)_{ij}(\sigma_u)_{ij} = (\epsv)_{ij}C_{ijkl}(\eps_u)_{kl}= (\epsv)_{kl} C_{ijkl} (\eps_u)_{ij} = (\eps_u)_{ij} (\sigvv)_{ij},
\end{align*}
which gives us
$$\elc(\mathbf u, \bm\varphi )= \int_{\mathbb{R}^2\backslash \Gamma} (\epsv)_{ij}(\sigma_u)_{ij}   ~\ud x.$$
Therefore, noticing  $\bm \varphi$ has symmetric properties \eqref{bcphi} due to the elastic extension,  integration by parts yields
\begin{align*}
  \elc(\mathbf u, \bm\varphi )=& \int_{\mathbb{R}^2\backslash \Gamma} (\nabla\bm\varphi )_{ij}(\sigma_u)_{ij}   ~\ud x\\
  =& \int_{\mathbb{R}^2\backslash \Gamma} -(\nabla\cdot\sigma_u)\cdot  \bm\varphi  ~ \ud x + \int_{y=0^+} ((\sigma_u)_{ij}n_j\varphi_i)^+ ~ \ud x+   \int_{y=0^-} ((\sigma_u)_{ij}n_j\varphi_i)^- ~ \ud x\\
  =& \int_\Gamma [-(\sigma_{u,12})^+ - (\sigma_{u,12})^- ] \varphi_1|_{\Gamma} ~ \ud x\\
  =& \int_\Gamma 2[(-\pt_{xx})^{1/2} u_1]\varphi_1 ~ \ud x = \eggc(u_1, \varphi_1),
\end{align*}
In the last equality, we used the relation in Lemma \ref{lem2.2} $\sigma_{12}^+  +\sigma_{12}^- = -\frac{2G}{(1-\nu) } (-\pt_{xx})^{\frac{1}{2}}u_1^+$ with the physical constant $\frac{G}{(1-\nu)}=1$.
 Thus we obtain \eqref{cross} and complete the proof of this lemma.
\end{proof}

Now combing Lemma \ref{reducelem} and Lemma \ref{lem2.8}, we give the proof of Theorem \ref{energy-re}.
\begin{proof}[Proof of Theorem \ref{energy-re}]
Notice the Taylor expansion of $\well$ at $u_1(\pm\8)=\mp\frac{b}{4}$ and Proposition \ref{thm2.3} (ii). It is easy to check the misfit energy $E_{\mathrm{mis}}$ in \eqref{Emis} is always finite. Thus, if $u_1|_\Gamma\in {H}^{\frac12}(\mathbb{R})$, by Lemma \ref{reducelem} for the elastic part in the total energy, we conclude part (i) of Theorem \ref{energy-re}. If $u_1|_\Gamma \notin \dot{H}^{\frac12}(\mathbb{R})$, by Lemma \ref{lem2.8}, we conclude part (i) of Theorem \ref{energy-re}.
\end{proof}

\subsection{Static solution is a global minimizer of the full system }\label{sec2.6}
In this section, we will prove  static solutions $\mathbf u$ (unique upto translations) obtained in Theorem \ref{thm_steady} are the global minimizers of the full system in the sense of Definition \ref{minimizer}.


Assume $\mathbf u\in  \ff_\Gamma (\mathbb{R}^2)$ for some $s>\frac12$ is the static solution obtained in Theorem \ref{thm_steady}, then we first  show that $\mathbf u$ is a  minimizer of $E$ in the sense  of Definition \ref{minimizer} for perturbations in $[-\frac{b}{4}, \frac{b}{4}]$; see Proposition \ref{prop3.3}. Then we will remove this constrain later in Theorem \ref{thm_mini}. Notice  Definition \ref{minimizer} for the global minimizer is in terms of all the perturbations with compact support. Since $\mathbf u\in  \ff_\Gamma (\mathbb{R}^2)$ for any $s>\frac12$, $\het$ is continuous in $H^{\frac32-s}(\mathbb{R}^2\backg)$ w.r.t $\bm\varphi $. Notice also the function space for perturbations $\bm\varphi $ in Definition \ref{minimizer} is dense in $H^1(\mathbb{R}^2\backg)\hookrightarrow H^{\frac32-s}(\mathbb{R}^2)$ with symmetry \eqref{bcphi}. It is easy to check the global minimizer defined  in Definition \ref{minimizer} can be equivalently generalized to any perturbations $\bm\varphi \in H^1(\mathbb{R}^2\backg; \mathbb{R}^2) $ with $\varphi_1|_{\Gamma}\in H^{\frac12}(\mathbb{R})$ satisfying \eqref{bcphi}.

The idea is to  use the  elastic extension in Theorem \ref{def_els} and the connections between the energy to the full system and the reduced energy to the slip plane in Theorem \ref{energy-re}.
In \cite[Theorem 1.4]{XC2005}, \textsc{Cabr\'e and Sol\`a-Morales} proved that a static solution to the scalar model from harmonic extension is a minimizer of the corresponding total energy  relative to perturbations in $[-\frac{b}{4}, \frac{b}{4}]$. In order to apply this result, we restate it in the setting of the reduced model \eqref{Gamma_eq} on $\Gamma$ below. 

\begin{prop}(\cite[Theorem 1.4]{XC2005})\label{mini-2005}
Assume $u_1|_\Gamma\in \dot{H}^s(\mathbb{R})$ for some $s>\frac12$ is a static solution obtained in Proposition \ref{thm2.3} (i). Given 
 any perturbations $\varphi_\Gamma\in H^\frac{1}{2}(\mathbb{R})$ such that $-\frac{b}{4}\leq(\varphi_\Gamma+ u_1|_{\Gamma})\leq -\frac{b}{4}$, we have 
\begin{equation}
\hat{E}_\Gamma(\varphi; u)\geq 0.
\end{equation}
\end{prop}
\begin{proof}
First,
let  $u\in C^2(\mathbb{R}^2_+)$ be the harmonic extension of $u_1|_\Gamma$ and $\varphi\in C^2(\mathbb{R}^2_+)$ be the harmonic extension of $\varphi_\Gamma$. Then by maximal principle for Laplace equation, $-\frac{b}{4}\leq \varphi+u \leq \frac{b}{4}$ in $ \mathbb{R}^2_+.$ 

 Second, from \cite[Theorem 1.4]{XC2005}, we have for any $R>0$
\begin{equation}
E_{\mathrm{total}}(u; R):=\frac{1}{2}\int_{\mathbb{R}^2_+ \cap B(R)}|\nabla u|^2 \ud x \ud y + \int_{-R}^R W(u_1) \ud x \leq E_{\mathrm{total}}(u+\varphi; R)
\end{equation}
for any $\varphi\in C^2(\overline{\mathbb{R}^2_+})$ with compact support in $B_+(R)\cup \Gamma$  such that  $-\frac{b}{4}\leq\varphi+ u\leq \frac{b}{4}$. It is well known the harmonic extension of $u_1$ satisfies $-\pt_\nu u_1=\ptf u_1$ on $\Gamma$. Then from integration by parts, we obtain
\begin{align*}
&\frac12\int_{\mathbb{R}} \int_0^{+\8} |\nabla u+ \nabla \varphi|^2 \ud y \ud x-\frac12 \int_{\mathbb{R}} \int_0^{+\8} |\nabla u|^2 \ud y \ud x=  \int_{\mathbb{R}} \int_0^{+\8}\frac12|\nabla \varphi|^2+ \nabla u \nabla \varphi \ud y \ud x \\
=&   \int_{\mathbb{R}} \int_0^{+\8}\frac12|\nabla \varphi|^2-\Delta u \varphi \ud y \ud x - \int_{\mathbb{R}} \pt_\nu u_1 \varphi_\Gamma \ud x= \int_{\mathbb{R}} \int_0^{+\8}\frac12|\nabla \varphi|^2 \ud y \ud x +\int_{\mathbb{R}}\varphi_\Gamma\ptf u_1 \ud x \\
=& \int_{\mathbb{R}} \frac{1}{2} |(-\pt_{xx})^{\frac14} \varphi_\Gamma|^2 \ud x+  \int_{\mathbb{R}}\varphi_\Gamma\ptf u_1 \ud x = 2 \hat{E}_{\Gamma_e}(\varphi; u_1).
\end{align*}  
Therefore we have
\begin{align*}
0\leq &E_{\mathrm{total}}(u+\varphi; R)-E_{\mathrm{total}}(u; R)\\
= &
\frac12\int_{\mathbb{R}} \int_0^{+\8} |\nabla u+ \nabla \varphi|^2 \ud y \ud x-\frac12 \int_{\mathbb{R}} \int_0^{+\8} |\nabla u|^2 \ud y \ud x+
 \int_{-R}^R W(u_1+ \varphi) \ud x- \int_{-R}^R W(u_1) \ud x 
 \\
 =& \hat{E}_\Gamma(\varphi; u_1)
\end{align*}
for any $\varphi\in C^1(\overline{\mathbb{R}^2_+})$ with compact support in $B_+(R)\cup \Gamma$ such that $-\frac{b}{4}\leq(\varphi+ u)|_{\Gamma}\leq \frac{b}{4}$.  

Third, since $\hat{E}_\Gamma(\varphi; u)$ is continuous in $H^{\frac32-s}(\mathbb{R}^2)$ w.r.t $\varphi$ and for any $s>\frac12$  $C_c^2(\mathbb{R}^2)$  is dense in $H^1(\mathbb{R}^2)\hookrightarrow H^{\frac32-s}(\mathbb{R}^2)$, $\hat{E}_\Gamma(\varphi; u)\geq 0$ holds also for any perturbation $\varphi_\Gamma\in H^{\frac12}(\mathbb{R})$ such that $-\frac{b}{4}\leq(\varphi_\Gamma+ u_1|_{\Gamma})\leq -\frac{b}{4}$.
\end{proof}

Before proving a static solution is a global minimizer, we first 
show that given $\varphi_1|_\Gamma\in H^{\frac12}(\mathbb{R})$,  its elastic extension yields a minimizer of the elastic energy $\el$.

\begin{lem}\label{equiv_lem}
Given $\varphi_1|_\Gamma \in H^{\frac12}(\mathbb{R})$, and its elastic extension $\bm \varphi$, then $\bm \varphi$ a minimizer of the elastic energy $E_{\mathrm{els}}$ with trace $\varphi_1|_\Gamma$ in the sense that   $\el(\bm \varphi)\leq \el( \bar{\bm \varphi})$ for any $\bar{\bm \varphi}\in H^1(\mathbb{R}^2\backg)$ satisfied \eqref{bcphi} with the same trace $\varphi_1|_\Gamma$.
\end{lem}
\begin{proof}
Since $\varphi_1|_\Gamma \in H^{\frac12}(\mathbb{R})$, same as \eqref{tm-p1}, we directly calculate that
\begin{equation}
\el(\bar{\bm \varphi}) - \el(\bm \varphi) = \el(\bar{\bm \varphi}- \bm \varphi) + \elc (\bar{\bm \varphi}-\bm \varphi, \bm \varphi).
\end{equation}
Notice the trace of $\bar{\bm \varphi}$ and the trace of $\bm \varphi$ are same. Using $\bm \varphi$ is the elastic extension of $\varphi_1|_\Gamma$ and the symmetry \eqref{bcphi} for $\bar{\bm \varphi}$, we have
\begin{align*}
  \elc(\bar{\bm \varphi}-\bm \varphi, \bm\varphi )=& \int_{\mathbb{R}^2\backslash \Gamma} (\nabla(\bar{\bm\varphi }-\bm \varphi))_{ij}(\sigma_\varphi)_{ij}   ~\ud x\\
  =& \int_{\mathbb{R}^2\backslash \Gamma} -(\nabla\cdot\sigma_\varphi)\cdot (\bar{\bm\varphi }-\bm \varphi)  ~ \ud x \\
  &+ \int_{y=0^+} ((\sigma_{\varphi})_{ij}n_j(\bar{\varphi}-\varphi)_i)^+ ~ \ud x+   \int_{y=0^-} ((\sigma_\varphi)_{ij}n_j(\bar{\varphi}-\varphi)_i)^- ~ \ud x\\
  =& \int_\Gamma [-(\sigma_{\varphi,12})^+ - (\sigma_{\varphi,12})^- ] (\bar{\varphi}-\varphi)_1|_{\Gamma} ~ \ud x=0.
\end{align*}
Then $\el(\bar{\bm \varphi}- \bm \varphi)\geq 0$ implies $\el(\bar{\bm \varphi}) - \el(\bm \varphi)\geq 0.$
\end{proof}

\begin{prop}\label{prop3.3}
 Let $\mathbf u\in  \ff_\Gamma (\mathbb{R}^2)$ for some $s>\frac{1}{2}$ be a static solution obtained in Theorem \ref{thm_steady}. Given any perturbations $\bm\varphi \in H^1(\mathbb{R}^2\backg; \mathbb{R}^2) $ with $\varphi_1|_{\Gamma}\in H^{\frac12}(\mathbb{R})$ satisfying \eqref{bcphi}   and $-\frac{b}{4}\leq(\varphi_1+u_1)|_\Gamma\leq \frac{b}{4}$, then we know $\mathbf u$ is a minimizer of $E$ such that
  \begin{equation}
        0\leq  \heg=\hetpsi \leq  \het,
  \end{equation}
  where $\bm\psi$ is the elastic extension of $\varphi_1|_\Gamma.$
\end{prop}
\begin{proof}
%
  First, since $\psi$ is the elastic extension of $\varphi_1|_\Gamma$, Theorem \ref{energy-re}(ii)  shows that
  \begin{equation}\label{com1}
   \hetpsi=\heg.
  \end{equation}

  Second, for any perturbation $\bm\varphi \in H^1(\mathbb{R}^2\backg; \mathbb{R}^2) $ with $\varphi_1|_{\Gamma}\in H^{\frac12}(\mathbb{R})$ satisfying \eqref{bcphi}, since $\psi$ is the elastic extension of $\varphi_1|_\Gamma$, we know $(\psi_1-\varphi_1)|_\Gamma=0$ and $(\bm \psi-\bm \varphi)\in H^1(\mathbb{R}^2)$. Therefore by Lemma \ref{equiv_lem}, we know 
   $${E_\mathrm{els}(\bm\psi)\leq E_\mathrm{els}}(\bm\varphi).$$
   Notice also $\eggc(u_1,\psi_1)=\eggc(u_1,\varphi_1)$, which together with the relation  \eqref{tm-p1}, leads to
   $$\helpsi\leq \hel.$$
Therefore
  \begin{equation}\label{com3}
  \begin{aligned}
    &\hetpsi
    =\helpsi+ \int_\Gamma \well(u_1+\varphi_1)-\well(u_1) \ud x\\
     \leq& \hel+ \int_\Gamma \well(u_1+\varphi_1)-\well(u_1) \ud x\\
    =&\het.
    \end{aligned}
  \end{equation}

  Finally, we apply Proposition \ref{mini-2005} to obtain
  $$0\leq \heg.$$
  This, together with \eqref{com1} and \eqref{com3}, yields
  \begin{equation}
    0\leq  \heg=\hetpsi \leq  \het.
  \end{equation}
  This completes the proof of Proposition \ref{prop3.3}.
\end{proof}
Similar to Proposition \ref{mini-2005}, Proposition \ref{prop3.3} above also requires 
all the  perturbations are between $[-\frac{b}{4},\frac{b}{4}]$.
The next theorem develops a new method to show the static solution for the PN model obtained in Theorem \ref{thm_steady} is a global minimizer in the sense of Definition \ref{minimizer} by removing the restriction that perturbation must be in the range of $[-\frac{b}{4},\frac{b}{4}]$.

{ \begin{thm}\label{thm_mini}
Let $\mathbf u\in  \ff_\Gamma (\mathbb{R}^2)$ for some $s>\frac12$ be a static solution obtained in Theorem \ref{thm_steady}. Then $\mathbf u$ is a global minimizer of $E$ in the sense that for any $\bm\varphi \in H^1(\mathbb{R}^2\backg)$ such that $\varphi_1|_\Gamma\in H^{\frac12}(\mathbb{R})$ the perturbation energy satisfies
  \begin{equation}\label{2.55}
    \het\geq 0.
  \end{equation}
  Besides,  if $\mathbf u$ is a global minimizer of $E$ then its trace $u_\Gamma$ is also a global minimizer of $E_\Gamma$. Conversely, if $u_\Gamma$ is a global minimizer of $E_\Gamma$ then its elastic extension $\mathbf u$ is a global minimizer of $E$. 
\end{thm}
\begin{proof}
  First, from Theorem \ref{energy-re} and \eqref{com3}, we know for any perturbation $\varphi_1|_{\Gamma}\in H^{\frac12}(\mathbb{R})$
\begin{equation}\label{ttt}
   \heg=\hetpsi\leq \het,
\end{equation}    
  where $\bm\psi$ is the elastic extension of $\varphi_1|_\Gamma.$ 
	Hence  it is sufficient to show that
	\begin{equation}
		 \heg\geq 0,
	\end{equation}
	where
	$\heg= \int_\Gamma |(-\pt_{xx})^{1/4} (u_1+ \varphi_1)|^2 -  |(-\pt_{xx})^{1/4} u_1|^2 \ud x+\int_\Gamma \well(u_1+\varphi_1)-\well(u_1) \ud x$ defined in \eqref{perEgamma}.
We have proved in Proposition \ref{prop3.3} that $\mathbf u$ is a minimizer for all the perturbations satisfying $-\frac{b}{4}\leq (\varphi_1+u_1)|_{\Gamma}\leq \frac{b}{4}$. For the case $  (\varphi_1+u_1)|_{\Gamma}$ not in $[-\frac{b}{4},\frac{b}{4}]$,  we  prove the same result using the method of contradiction below.

Suppose that $\heg<0$ for some $ \bm\varphi  \in C^\infty(\mathbb{R}^2\backslash \Gamma; \mathbb{R}^2)$ and $ \bm\varphi $ has compact support in some $B(R)$ satisfying symmetry \eqref{bcphi}. Let $v:=(\varphi_1+u_1)|_\Gamma$, and a cut off function $\bar{v}:=\max\{\min\{v,\frac{b}{4}\},-\frac{b}{4}\}$. Since $u_1$ is monotone and connect from $\frac{b}{4}$ to $-\frac{b}{4}$, $\bar{\varphi}_1|_\Gamma:=\bar{v}-u_1$ still has compact support in $B(R)$. Denote  $\bar{\bm\varphi }$ as the elastic extension of $\bar{\varphi}_1|_\Gamma$.  Note that $|\bar{v}(x)-\bar{v}(x')|\leq |v(x)-v(x')|$ for any $x,x'\in\mathbb{R}$. Thus
    \begin{align}
    	\frac{1}{2} \int_\Gamma |(-\pt_{xx})^{1/4}\bar{v}|^2-|(-\pt_{xx})^{1/4}u_1|^2\ud x
    	&= \frac{1}{4\pi}\int_{\mathbb{R}}\int_{\mathbb{R}}\frac{|\bar{v}(x)-\bar{v}(x')|^2}{|x-x'|^{2}}-\frac{|u_1(x)-u_1(x')|^2}{|x-x'|^{2}}\ud x \ud x'\nonumber\\
        &\leq  \frac{1}{4\pi}\int_{\mathbb{R}}\int_{\mathbb{R}}\frac{|v(x)-v(x')|^2}{|x-x'|^{2}}-\frac{|u_1(x)-u_1(x')|^2}{|x-x'|^{2}}\ud x \ud x'\nonumber\\
        &= \frac{1}{2} \int_\Gamma |(-\pt_{xx})^{1/4}{v}|^2-|(-\pt_{xx})^{1/4}u_1|^2\ud x.\label{eq..Eu.Eubar.part1}
    \end{align}
    Also note that for $v(x)\geq \frac{b}{4}$ or $v(x)\leq-\frac{b}{4}$, $\bar{v}(x)=\pm \frac{b}{4}$ and $\well(\bar{v}(x))=0\leq \well(v(x))$. Thus
    \begin{equation}
    	\int_\Gamma \well(\bar{v})-\well(u_1)\ud x
    	\leq \int_\Gamma \well(v)-\well(u_1)\ud x.
    	\label{eq..Eu.Eubar.part2}
    \end{equation}
    Combining Eqs. \eqref{eq..Eu.Eubar.part1} and \eqref{eq..Eu.Eubar.part2}, we immediately obtain $\hat{E}_\Gamma(\bar{\bm\varphi}; \mathbf u)\leq \heg<0$. On the other hand, Proposition \ref{prop3.3} implies that $0\leq \hat{E}_\Gamma(\bar{\bm\varphi};\mathbf u)$ since $-\frac{b}{4}\leq \bar{v}= (\bar{\varphi}_1+u_1)|_\Gamma\leq \frac{b}{4}$. This contradiction completes the proof of \eqref{2.55}.
 
    Finally, we clarify the relation between the minimizer of the full system and the minimizer of the reduced system. On one hand, from \eqref{ttt}, $\heg\geq 0$ implies $\het\geq 0$. On the other hand, $\hetpsi\geq 0$ implies $\heg\geq 0$.
\end{proof}
}

\section{Global classical solution to dynamic PN model}\label{sec3}
In this section, we consider the dynamic model with the total energy $E$ in \eqref{E2.2}. Here we focus on the dynamics of a dislocation structure and neglect the inertia effect of the materials. In other words, we consider the overdamped regime, which is a gradient flow of the total energy. This is reasonable since the dislocation dynamics on the slip plane $\Gamma$ has a much larger time scale than the relaxation time of the elastic parts. Hence we take a quasi-static assumption for the upper/lower half space $y>0$ and $y<0$, i.e., $\partial_t u=0$ in $\mathbb{R}^2\backslash \Gamma.$ Indeed the quasi-static assumption leads to a homogenous elastic equation in the upper/lower half space $y>0$ and  $y<0$, which is the key point to establish the relation between the full system and the reduced system in terms of solutions as well as energies. 

Recall the free energy $E_\Gamma$ on the slip plane is
\begin{equation}
E_\Gamma(u_1)=\int_\Gamma |(-\pt_{xx})^{\frac{1}{4}}u_1|^2 \ud x +\int_\Gamma \well(u_1) \ud x;
\end{equation}
see the specific definition for the perturbed energy in \eqref{heg-v}. After the quasi-static approximation,  we can use the elastic extension in Theorem \ref{def_els} to see that a solution to the dynamic system on the slip plane $\Gamma$ gives naturally the displacement fields in the full space.
 In other words, from the relation between the trace $u_1|_{\Gamma}$ and solution $\mathbf u$ in the full space stated in Theorem \ref{thm_steady}, the dynamic model  becomes an elliptic problem with a nonlinear dynamic boundary condition
\begin{equation}\label{P2}
\begin{aligned}
&\nabla \cdot \sigma =0 \quad\text{ in } \mathbb{R}^2\backslash \Gamma,\\
&\partial_t u_1 =-2(-\pt_{xx})^{\frac{1}{2}} u_1 -\well'(u_1) \quad\text{ on }\Gamma,\\
&\sigma_{22}^+ =\sigma_{22}^- \quad\text{ on }\Gamma.
\end{aligned}
\end{equation}
  We also provide explanations using a gradient flow for the full system with different mobilities in Remark \ref{rem2}. Here and in the following, we set some physical constants  to be $1$ for simplicity.

Our main goal in this section is to prove the uniqueness and existence of the classical solution to problem \eqref{P2} with boundary conditions \eqref{symsym}, \eqref{BC} and initial data $u_0$.

Notice the nonliearity $\well(\cdot)$ effects only the first variable $u_1$ and thus by the elastic extension of $u_1|_\Gamma$ we can determine uniquely the solution to Problem \eqref{P2} as long as we can solve $u_1$  on $\Gamma$. We focus on the one dimensional nonlocal equation
\begin{equation}\label{main3.1}
\partial_t u_1 + 2(-\pt_{xx})^{\frac{1}{2}}u_1+\well'(u_1)=0, \quad x\in \mathbb{R}
\end{equation}
with boundary condition
\begin{equation}\label{bc3.6}
u_1(+\infty)=-1; \quad u_1(-\infty)=1.
\end{equation}
We remark the boundary condition here is well-defined since in the end we obtain the dynamic solution $u_1$ in the classical sense by proving the perturbation $v=u_1-u_1^*\in C((0,\8); H^1(\mathbb{R}))$, where $u_1^*$ is the static solution to the reduce model \eqref{Gamma_eq}.

Recall the free energy $E_\Gamma$ for the reduced model  is infinity. As in the last section,  we still use
the perturbated total energy on $\Gamma$ with respect to the trace $u_1^*|_\Gamma$ of the static solution $\mathbf u^*$ obtained in Theorem \ref{thm_steady}
\begin{equation}\label{heg-v}
\hat{E}_{\Gamma}(v; \mathbf u^*)= \int_\Gamma |(-\pt_{xx})^{1/4} u_1|^2 -  |(-\pt_{xx})^{1/4} u_1^*|^2 ~ \ud x +\int_\Gamma \well(u_1)-\well(u_1^*) \ud x,
\end{equation}
which is equivalent to
\begin{align}\label{eqvEn}
\hegu=&  \int_\Gamma |(-\pt_{xx})^{1/4} (u_1-u_1^*)|^2  ~ \ud x+2\int_\Gamma (-\pt_{xx})^{1/4} u_1^* (-\pt_{xx})^{1/4}(u_1-u_1^*) ~\ud x+\int_\Gamma \well(u_1)-\well(u_1^*) \ud x\\
=& \int_\Gamma |(-\pt_{xx})^{1/4} (u_1-u_1^*)|^2  ~ \ud x-\int_\Gamma \well'( u_1^*)  (u_1-u_1^*) ~\ud x+\int_\Gamma \well(u_1)-\well(u_1^*) \ud x,\nonumber\\
=&\int_{\Gamma} |(-\pt_{xx})^{1/4} v|^2 -v\well'(u_1^*) + \well(v+u_1^*)-\well(u_1^*) \ud x,
\nonumber
\end{align}
due to $u^*_1$ is the static solution satisfying \eqref{Gamma_eq}.
Thus the  reduced system on $\Gamma$ has its own gradient flow structure
\begin{equation}\label{E_G_d}
\partial_t u_1 = -\frac{\delta \hegu}{\delta u_1}.
\end{equation}
{{In the following subsection, we will establish the global classical solution to the perturbation $v=u_1-u_1^*$, which is the difference between $u_1$ and the static solution $u_1^*$.  }}

{
\begin{rem}\label{rem2}
We can also explain the quasi-static assumption by a gradient flow with different mobilities.
In general, for a over-damped dynamical system, the governing equation is given by $V=Mf$, where $V$ is the time derivative of parameters of the state, $f$ is the (negative) variation of the free energy, and $M$ is the corresponding mobility which is basically the reciprocal of the damping coefficient. For a crystalline solid with dislocations, the mobility is not homogeneous in the sense that it has different magnitude in the elastic continua and on the slip plane, denoted as $M$ and $M_\Gamma$ respectively. Experimental observations show that $M\gg M_\Gamma$ for most dislocations. In the following we assume $M=O(1/\eps)$ and $M_\Gamma=O(1)$ where $\eps$ is a small parameter.
Notice all the variables are dimensionless in our paper.
Instead of defining $E$ as \eqref{E2.2}, we use another independent variable $\tilde{\mathbf u}$ to indicate the fast bulk variable and consider gradient flow in the bulk and on the $\Gamma$ separately. Notice  the elastic extension of $u_\Gamma$ is unique. After taking limit $\eps\to 0$, we will see the trace of the independent variable $\tilde{\mathbf u}$ is indeed consistent with $u_\Gamma$, i.e. $\tilde{u}_1(x,0)=:\tilde{u}|_{\Gamma}=u_\Gamma:=u_1(x,0)$.   Define
$$E(\tilde{u}, u_{\Gamma}):= \int_{\mathbb{R}^2\backslash \Gamma} \frac{1}{2} \tilde{\eps}: \tilde{\sigma} \ud x \ud y +\int_{\Gamma}\well(u_{\Gamma})\ud x, $$
where $\tilde{\eps}$ and $\tilde{\sigma}$ are the corresponding stain and stress tensor of $\tilde{\mathbf u}$. We consider the gradient flow w.r.t $\tilde{\mathbf u}$ and $u_{\Gamma}$
\begin{equation}
  \frac{\ud }{\ud t}\left( \begin{array}{c}
                             \tilde{u} \\
                             u_\Gamma
                           \end{array}
   \right)= - \left[ \begin{array}{cc}
                       M & 0 \\
                       0 & M_\Gamma
                     \end{array}
    \right] \frac{\delta E(\tilde{u}, u_\Gamma)}{\delta (\tilde{u}, u_\Gamma)}.
\end{equation}
Similar to \eqref{tem2.8}-\eqref{tem2.10}, 
the gradient flow turns out to be
\begin{align*}
 &\frac{1}{M}\partial_t \tilde{u}=- \nabla \cdot \tilde{\sigma},\quad x\in \mathbb{R}^2\backslash \Gamma,\\
 &\frac{1}{M_\Gamma}\partial_t u_{\Gamma} = - [\tilde{\sigma}_{12}^{+}+\tilde{\sigma}_{12}^{-}+\well'(u_\Gamma)], \quad x\in \Gamma,\\
 &\tilde{\sigma}_{22}^{+}=\tilde{\sigma}_{22}^{-}, \quad x\in \Gamma.
\end{align*}
Let $\eps\to 0$, $\frac{1}{M} \to 0$, which indicates $- \nabla \cdot \tilde{\sigma}=0$ then $\tilde{\mathbf u}$
is the elastic extension of $u_\Gamma$ and coincides with $\mathbf u$ . Thus Lemma \ref{lem2.2} part (ii) shows that
$$\tilde{\sigma}_{12}^{+}=\tilde{\sigma}_{12}^{-}= 2(-\pt_{xx})^{\frac{1}{2}}u_{\Gamma}, \quad x\in\Gamma.$$
We obtain the reduced dynamic system on $\Gamma$ \eqref{main3.1}.

\end{rem}

}

\subsection{Global classical solution}
In this section, we will use the theory for analytic semigroup to  establish the existence and uniqueness of the global classical solution to \eqref{main3.1} by studying the existence and uniqueness in terms of the perturbation fields.
In terms of the reference field $\mathbf u^*$ such that $2\ptf u_1^* = -\well'(u_1^*)$, set the perturbation $v(x,t):= u_1(x,t)-u_1^*(x)$.
Then from the dynamic equation \eqref{main3.1}, we know the dynamic equation for $v$ is
\begin{equation}\label{v-eq}
  \partial_t v = -2\ptf v -\well'(v+u_1^*)+ \well'(u^*_1)
\end{equation}
with initial data $v_0(x)=u_1(x,0)-u^*_1(x)$.
Denote $H^s(\mathbb{R})$ as the (fractional) Sobolev space with norm denoted as $\|\cdot\|_s$. 
Denote $\|\cdot\|$ as the standard $L^2(\mathbb{R})$ norm.

Define the free energy for $v$ as
\begin{equation}
 \F(v):= \int_{\Gamma} |(-\pt_{xx})^{1/4} v|^2 -v\well'(u_1^*) + \well(v+u_1^*) \ud x.
\end{equation}
Notice this energy differs with \eqref{eqvEn} with a term $\int_\Gamma \well(u_1^*) \ud x$ whose variation is $0$.
Then $v$ satisfies the gradient flow structure
$$\partial_t v = -\frac{\delta \F(v)}{\delta v}. $$
Define
  \begin{equation}\label{A2.3}
    A v := (\ptf+I) v,
  \end{equation}
  \begin{equation}\label{TT}
    T( v) := \well'(u_1^*)-\well'(v+u_1^*)+v.
  \end{equation}
Then the \eqref{v-eq} becomes
  \begin{equation}\label{v-eq-n}
    \partial_t   v = -A v + T(v).
  \end{equation}
Since the spectrum for $A$ is $\sigma(A)=[1,+\8),$ from \cite[Definition 1.3.1]{Henry}, $A$ is a sectorial operator from $D(A)=H^1(\mathbb{R})\subset L^2(\mathbb{R})\to L^2(\mathbb{R})$ in the sense that
$$S_{1,\beta}:= \{\lambda\,|\, \beta \leq |\arg(\lambda-1)|\leq \pi, \, \lambda\neq 1 \}$$
is in the resolvent set of $A$ and
\begin{equation}
  \|(\lambda-L_s)^{-1}\|\leq \frac{1}{|\lambda-1|} \quad \text{ for all } \lambda\in S_{1, \beta}.
\end{equation}

The existence and uniqueness of the global classical solution to \eqref{v-eq} is stated as follows.
 \begin{thm}\label{strongslu}
 Assume initial data $v_0(x):=u_0(x)-u_1^*(x)\in H^{\frac{1}{2}}(\mathbb{R}).$
 \begin{enumerate}[(i)]
    \item There exists a global unique solution
 \begin{equation}
   v\in C^1([0,\8); L^2(\mathbb{R}))\cap C((0,\8); H^1(\mathbb{R})) \quad
 \end{equation}
to \eqref{v-eq-n} such that $v(x,0)=v_0(x)$ and $\partial_t v , Av, T(v)\in L^2(\mathbb{R})$ for $t>0$ and the equation \eqref{v-eq-n} is satisfied in $L^2(\mathbb{R})$ for any $t>0$;
    \item the solution can be expressed by
\begin{equation}\label{mild}
  v(t) = e^{-At}v_0 + \int_0^t e^{-A(t-\tau)} T(v(\tau)) \ud \tau;
\end{equation}
    \item for any $k,j\in \mathbb{N}^+$ and $\delta >0$ there exist constants $c,\, C_{\delta, k, j}$ such that
\begin{equation}\label{highreg}
\begin{aligned}
  v\in C^k((0,\8);H^j(\mathbb{R}));\\
  \|\pt_t^k v(\cdot, t) \|_j \leq C_{\delta, k,j} e^{c t}, \quad  t\geq \delta;
  \end{aligned}
\end{equation}
\item we have energy identity
\begin{equation}\label{dissi}
 \frac{\ud \F(v(t))}{\ud t} =  - \int_{\mathbb{R}} [-(-\pt_{xx})^{1/2} v -\well'(v+u_1^*)+\well'(u_1^*)]^2 \ud x  =: -\mathcal{Q}(v(t))\leq 0,
\end{equation}
and furthermore, if for misfit energy $E_{\mathrm{mis}}$ defined in \eqref{Emis}, the initial data $v_0(x)$ satisfies $E_{\mathrm{mis}}(v_0+u_1^*)<\8$, we have
\begin{equation}
\F(v(t))\leq \F(v_0), \quad \text{ for any }t\geq 0.
\end{equation}
 \end{enumerate}
\end{thm}
\begin{proof}
Step 1. We state some properties for $T$ defined in \eqref{TT}.
From \cite[Theorem 1.6]{XC2005} we  know the static solution
 $$|1+u_1^*|\leq \frac{c}{1+|x|} \text{ for }x>0, \quad |1-u_1^*|\leq \frac{c}{1+|x|} \text{ for }x<0,  $$
 and
 $$|\partial_x u_{1}^{*}|\leq \frac{c}{1+x^2},$$
  which shows $\|\partial_x u_{1}^{*}\|<+\infty$.
Then we have

(a) $T: L^2(\mathbb{R})\to L^2(\mathbb{R})$  is global Lipschiz, i.e.  there exists a constant $L$ such that
\begin{equation}\label{GLip}
\|T(v_1)-T(v_2)\|\leq (1+\max|\well'|)\|v_1-v_2\|\leq L\|v_1-v_2\|;
\end{equation}

(b) if $v(\cdot)\in H^1(\mathbb{R})$, then $T(v(\cdot))\in H^1(\mathbb{R})$. Indeed,
\begin{equation*}
 \|\pt_x T(v)\|\leq (1+\max|\well''|)\|v_x\|+ \pi \|v\|,
\end{equation*}
which implies
\begin{equation}
  \|T(v)\|_1 \leq c\|v\|_1.
\end{equation}

Step 2.
 Firstly, it is easy to check that the operator $A$ defined in \eqref{A2.3} is m-accretive in $L^2(\mathbb{R})$. Indeed we know $\mathrm{Re}\la Ax, x\ra\geq 0$ for all $x\in D(A)$ and $\sigma(A)=[1,+\8).$ Therefore $A$ is an infinitesimal generator of a linear strongly continuous semigroup of contractions and $\|e^{-At}\|\leq 1$.
  Secondly, from global Lipschitz condition \eqref{GLip},  there exists a unique mild solution
  expressed by  \eqref{mild} and $v\in C([0,+\8); L^2(\mathbb{R}))$.

  Step 3.
 H\"older continuity in $t$ of $v$ and $T(v)$.
  \begin{align}\label{tma3}
    &v(t+h)-v(t)\\
    =&e^{-At}(e^{-Ah} v_0 - v_0)+ \int_0^{t+h} e^{-A(t+h-\tau)}T(v(\tau))\ud \tau- \int_0^t e^{-A(t-\tau)} T(v(\tau)) \ud \tau\nonumber\\
    =&e^{-At}\big[(e^{-Ah} v_0 - v_0)+\int_0^h e^{-A(h-\tau)}T(v(\tau)) \ud  \tau\big]+\int_0^t e^{-A(t-\tau)} [T(v(\tau+h))-T(v(\tau))] \ud \tau\nonumber\\
    =&e^{-At}(v(h)-v_0)+ \int_0^t e^{-A(t-\tau)} [T(v(\tau+h))-T(v(\tau))] \ud \tau \nonumber
  \end{align}
  Since $\|e^{-At}\|\leq 1$,
    \begin{align*}
    \|v(t+h)-v(t)\|
    \leq \|v(h)-v_0\| +  \int_0^t 2\|v(\tau+h)-v(\tau)\|\ud \tau.
  \end{align*}
  Then by Gronwall's inequality, we have
  \begin{equation}\label{vLip_0}
    \|v(t+h)-v(t)\|\leq \|v(h)-v_0\|e^{2t}.
  \end{equation}
On the other hand,
\begin{align}\label{At2}
 {v(h)-v_0} = (e^{-Ah}-I)v_0 +  \int_0^h e^{-A(h-\tau)} [T(v(\tau))-T(v_0)+T(v_0) ]\ud \tau.
\end{align}
Then from \eqref{GLip} and $\|e^{-At}\|\leq 1$ we know
\begin{align*}
  &\|v(h)-v_0\|\leq \|(e^{-Ah}-I)v_0\|+ L\int_0^h \|v(\tau)-v_0\| \ud \tau + hL \|v_0\| \\
  =& hL\|v_0\|+Ch^{1/2}\|A^{1/2}v_0\| + L \int_0^h \|v(\tau)-v_0\| \ud \tau,
\end{align*}
where we used the fact $A$ is sectorial and thus  from \cite[Theorem 1.4.3]{Henry}
$$\|(e^{-At}-I)v_0\|\leq C h^{\frac12}\|A^{1/2}v_0\|.$$

Thus Gronwall's inequality gives us
\begin{equation}
\|v(h)-v_0\|\leq h^{\frac12}(h^{\frac12}L\|v_0\|+C\|A^{\frac12}v_0\|)e^{Lh},
\end{equation}
which, together with \eqref{vLip_0}, leads to the H\"older continuity of $v(t)$
\begin{equation}\label{vLip}
\left\| \frac{v(t+h)-v(t)}{h^{\frac12}} \right\|\leq c \|v_0\|_{\frac12} e^{2t+Lh}.
\end{equation}
Then from \eqref{GLip} we concludes the H\"older continuity of $T(v(t))$
\begin{equation}\label{GLip_t}
\left\| \frac{T(v(t+h))-T(v(t))}{h^{\frac12}} \right\|\leq  c \|v_0\|_{\frac12} e^{2t+Lh}.
\end{equation}
Therefore by \cite[Lemma 3.2.1]{Henry} we know for $t>0$
\begin{equation}
\int_0^t e^{-A(t-\tau)} T(v(\tau)) \ud \tau \in D(A).
\end{equation}
Notice also
$$\|Ae^{-At}v_0\|\leq \frac{c}{t}e^{-t}$$
for $t>0$, which shows $e^{-At}v_0\in D(A)$ for $t>0.$
Therefore by mild solution \eqref{mild} we concludes $v \in D(A)$ and $\partial_t v =-Av+T(v)\in L^2$ for $t>0$, which completes the proof for (i), (ii).

  Step 4. Higher order regularities.

Set $w_1:=\pt_t v$ and $w_2:=\pt_x v$. Then
$$\pt_t T(v(t))=T'(v)\partial_t v\in C([0,T]; L^2(\mathbb{R}))$$
and
$$\pt_x T(v(t))=(1-\well'(u_1^*+v))\partial_x v-(\well'(u_1^*+v)-\well'(u_1^*))\partial_x u_{1}^*\in C([0,T]; L^2(\mathbb{R})).$$
Therefore we can repeat Step 2 and 3 for
\begin{equation}
  \partial_t w_{1}+ Aw_1 = T'(v)w_1
\end{equation}
and
\begin{equation}
  \partial_t w_{2}+Aw_2 = (1-\well'(u_1^*+v))w_2-(\well'(u_1^*+v)-\well'(u_1^*))\partial_xu_{1}^*
\end{equation}
to obtain
\begin{align*}
  &w_1, w_2 \in C((0,\8);L^2(\mathbb{R}))\cap C((0,\8); H^1(\mathbb{R}))\\
  &\partial_t w_{t}, \partial_t w_{t} \in C((0,\8); L^2(\mathbb{R}))
\end{align*}
which concludes $v$ is a global classical solution to \eqref{v-eq} and satisfies \eqref{highreg}.

Step 5. \eqref{dissi} is directly from \eqref{v-eq} and above regularity properties.
Notice that if the initial data $v_0(x)$ satisfies $E_{\mathrm{mis}}(v_0+u_1^*)<\8$, then from  $\|u^*_1(\cdot)\|<c$ and
$v_0(x)\in H^{\frac{1}{2}}(\mathbb{R})$ we have $\F(v_0)<\8$
and thus
$$\F(v(t))\leq \F(v_0)<\8.$$
\end{proof}

\section*{Acknowledgement}
The work of YG and YX was supported by the Hong Kong Research Grants Council General Research Fund 16313316. JGL was supported in part by the National Science Foundation (NSF) under award DMS- 1812573 and the NSF grant RNMS-1107444 (KI-Net).

\end{document}